\documentclass[11.5pt,reqno]{amsart}

\usepackage{amsmath}
\usepackage[T1]{fontenc}
\usepackage{mathrsfs}
\usepackage{latexsym}
\usepackage{amssymb} 
\usepackage{graphicx}
\usepackage{tikz}
\usepackage{float}
\usepackage{extarrows}
\usepackage{epstopdf}
\usepackage{caption}
\usepackage{subcaption}
\usepackage{xcolor}
\usepackage[a4paper, right=3cm,top=3cm, bottom=3cm]{geometry}
 \floatstyle{boxed}
 \floatstyle{plain}

\numberwithin{equation}{section}

\newcommand{\R}{\mathbb{R}}

\newcommand{\N}{\mathbb{N}}

\newcommand{\chookrightarrow}{\mathrel{\lhook\joinrel\relbar\kern-.8ex\joinrel\lhook\joinrel\rightarrow}}
\newcommand{\warrow}{\rightharpoonup}

\newcommand{\e}{\varepsilon}
\newcommand{\ei}{{\e^{-1}}}

\normalsize

\DeclareMathOperator{\spa}{span}

\DeclareMathOperator{\Id}{Id}

\newtheorem{satz}{Proposition}[section] 
\newtheorem{lem}[satz]{Lemma} 
\newtheorem{bem}[satz]{Remark}
\newtheorem{thm}[satz]{Theorem}

\newtheorem{kor}[satz]{Corollary}
\theoremstyle{remark}

\definecolor{gray}{gray}{0.50}
\definecolor{lred}{rgb}{1.0,0.5,0.5}
\definecolor{dgreen}{rgb}{0,1,1}
\definecolor{luh-dark-blue}{rgb}{0.0, 0.313, 0.608}

\title[Weak solutions to a two-phase thin film model with insoluble surfactant]{Weak Solutions to a Two-Phase Thin Film Model with Insoluble Surfactant Driven by Capillary Effects}
\author{Gabriele Bruell}
\address{Department of Mathematical Sciences, Norwegian University of Science and Technology, 7491 Trondheim, Norway}
\email{gabriele.bruell@math.ntnu.no}


\usepackage[colorlinks=true]{hyperref}
\hypersetup{urlcolor=luh-dark-blue, citecolor=lred, linkcolor=luh-dark-blue}

\begin{document}
\maketitle

\begin{abstract} 
Of concern is the study of a system of three equations describing the motion of a viscous complete wetting two-phase thin film endowed with a layer of insoluble surfactant on the surface of the upper fluid under the effects of capillary forces.
 The governing equations for the film heights of the two-phase flow are degenerate, parabolic and strongly coupled fourth-order equations, which are additionally coupled to a second-order parabolic transport equation for the surfactant concentration. A result on the existence of non-negative global weak solutions is presented.
\end{abstract}

\section{Introduction} 
\allowdisplaybreaks
Consider two immiscible, incompressible Newtonian and viscous thin liquid films on top of each other on a solid substrate. We assume that there is no contact angle between the two-phase flow and the bottom, which places the setting in the context of complete wetting. The interface of the upper fluid is endowed with a layer of insoluble surfactant. Surfactants act on the surface of a fluid  by lowering the surface tension and induce a  twofold dynamic. On the one hand, the resulting surface gradients influence the dynamics of the fluid film. On the other hand, the surfactants spread along the interface, which is called \emph{Marangoni effect}. 
Recently, a system describing the dynamics of a two-phase thin film with insoluble surfactant has been derived in \cite{B1}, by the method of lubrication approximation and cross-sectional averaging.
Considering capillary effects as the only driving force and neglecting gravitational as well as intermolecular (van der Waals) forces, the system we are studying is parabolic, degenerated, strongly coupled and  given by 
\begin{align}
\nonumber
&\partial_t f +\partial_x \left[ f\left(\displaystyle{\frac{Rf^2}{3}}\partial_x^3 f + S\mu \left(\displaystyle{\frac{f^2}{3}} +\displaystyle{\frac{fg}{2}} \right) \partial_x^3(f+g) +\mu 	\displaystyle{\frac{f}{2}}\partial_x\sigma(\Gamma)\right)\right]=0,\\[15pt]
\label{system2}
& \partial_t g + \partial_x \left[g\left( \displaystyle{\frac{Rf^2}{2}} \partial_x^3 f +S\left(  \displaystyle{\frac{g^2}{3}} +\mu \left( \displaystyle{\frac{f^2}{2}}+fg\right)\right)\partial_x^3 (f+g) 
+\left(\mu f +\displaystyle{\frac{g}{2}} \right) \partial_x \sigma(\Gamma)\right)\right]=0, \\[15pt]
\nonumber
&\partial_t \Gamma +\partial_x\left[\Gamma\left( \displaystyle{\frac{Rf^2}{2}} \partial_x^3 f +S\left(  \displaystyle{\frac{g^2}{2}}+\mu \left(\displaystyle{\frac{f^2}{2}}+fg\right)\right)\partial_x^3 (f+g)
+\left(\mu f+g\right) \partial_x \sigma(\Gamma)\right)- D\partial_x\Gamma\right]=0
\end{align}
in $\Omega_\infty:=(0,\infty)\times (0,L)$, with $\Omega_\infty$ being the time-space domain and the lateral boundary of the system is given at $x=0,L$. The unknowns are the functions $f=f(t,x)$ and $g=g(t,x)$ parameterizing the interfaces separating the fluids and the upper fluid from air, respectively, and the surfactant concentration $\Gamma=\Gamma(t,x)$. 
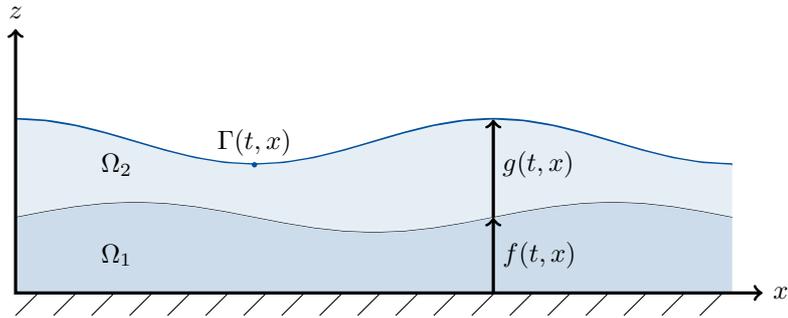
\begin{figure}[h]
\centering
\begin{tikzpicture}[domain=0:3*pi, scale=1] 
\draw[color=black] plot (\x,{0.3*cos(\x r)+2}); 
\draw[very thick, smooth, variable=\x, luh-dark-blue] plot (\x,{0.3*cos(\x r)+2}); 
\fill[luh-dark-blue!10] plot[domain=0:3*pi] (\x,{0.2*sin(\x r)+1}) -- plot[domain=3*pi:0] (\x,{0.3*cos(\x r)+2});
\draw[color=black] plot (\x,{0.2*sin(\x r)+1});
\fill[luh-dark-blue!20] plot[domain=0:3*pi] (\x,0) -- plot[domain=3*pi:0] (\x,{0.2*sin(\x r)+1});
\draw[very thick,<->] (3*pi+0.4,0) node[right] {$x$} -- (0,0) -- (0,3.5) node[above] {$z$};
\draw[very thick,->] (2*pi,1) -- (2*pi,2.3);
\node[right] at (2*pi,1.7) {$g(t,x)$};
\draw[very thick,->] (2*pi,0) -- (2*pi,1);
\node[right] at (2*pi,0.5) {$f(t,x)$};
\coordinate[label=above:{$\Gamma(t,x)$}] (A) at (pi,1.7);
\fill[color=luh-dark-blue] (A) circle (1pt);
\node[right] at (1,0.5) {$\Omega_1$};
\node[right] at (1,1.7) {$\Omega_2$};
\draw[-] (0,-0.3) -- (0.3, 0);
\draw[-] (0.5,-0.3) -- +(0.3, 0.3);
\draw[-] (1,-0.3) -- +(0.3, 0.3);
\draw[-] (1.5,-0.3) -- +(0.3, 0.3);
\draw[-] (2,-0.3) -- +(0.3, 0.3);
\draw[-] (2.5,-0.3) -- +(0.3, 0.3);
\draw[-] (3,-0.3) -- +(0.3, 0.3);
\draw[-] (3.5,-0.3) -- +(0.3, 0.3);
\draw[-] (4,-0.3) -- +(0.3, 0.3);
\draw[-] (4.5,-0.3) -- +(0.3, 0.3);
\draw[-] (5,-0.3) -- +(0.3, 0.3);
\draw[-] (5.5,-0.3) -- +(0.3, 0.3);
\draw[-] (6,-0.3) -- +(0.3, 0.3);
\draw[-] (6.5,-0.3) -- +(0.3, 0.3);
\draw[-] (7,-0.3) -- +(0.3, 0.3);
\draw[-] (7.5,-0.3) -- +(0.3, 0.3);
\draw[-] (8,-0.3) -- +(0.3, 0.3);
\draw[-] (8.5,-0.3) -- +(0.3, 0.3);
\draw[-] (9,-0.3) -- +(0.3, 0.3);
\end{tikzpicture} 
\caption{Scheme of the two-phase thin film flow with insoluble surfactant}
\end{figure}
Here, the material constant $\mu:=\frac{\mu_2}{\mu_1}$ is the relative viscosity, where $\mu_1$ and $\mu_2$ denote the viscosity of the lower and the upper fluid, respectively, and $D>0$ is the surface diffusivity of the surfactant.
We assume the surface tension coefficient  $\sigma=\sigma(\Gamma)$ to be decreasingly dependent on the surfactant concentration. The  constants  
\[R:= \sigma_1^c+\sigma_2^c\mu\qquad \mbox{and}\qquad S:=\sigma_2^c\]
contain the surface tension coefficients $\sigma_1^c,\sigma_2^c>0$ of the interface of the lower and the upper fluid, respectively, which are independent of the surfactant concentration.
Further, \eqref{system2} is supplemented by  initial data at $t=0$ for the three unknowns
\begin{equation} \label{wID}
f(0,\cdot)=f^0,\quad g(0,\cdot)=g^0, \quad \Gamma(0,\cdot)=\Gamma^0
\end{equation}
and boundary conditions at $x=0,L$
\begin{equation}\label{wNB}
\begin{array}{lll}
&\partial_xf=\partial_xg=\partial_x\Gamma =0,\\[10pt]
&\partial_x^3f=\partial_x^3g=0.
\end{array}
\end{equation}

The difficulty in studying system \eqref{system2} relies  in particular in the two sources of degeneracies, where the  film heights may vanish on subsets of $(0,L)$. The existence of local strong solutions to \eqref{system2} has been shown in \cite{B1}. Owing to the degeneracy, it is in general not clear whether one can prove the existence of global solutions in a classical sense, which motivates the study of weak solutions.

If $g$ and $\Gamma$ vanish both, then the system reduces to the famous thin-film equation
\[\partial_t f+\partial_x \left[ f^n\partial_x^3 f\right]=0 \qquad \mbox{with}\quad n=3,\]
for which weak solutions were constructed first in the pioneering work by Bernis and Friedmann \cite{BF}. 
 Various contributions have also been dedicated to a two-phase generalization of the thin film equation. The study of weak solutions for a two-phase thin film system without surfactant has been addressed in  \cite{JKT} $(n=2)$ and \cite{EM,JKT} ($n=3$)\footnote{$n=2$: two-phase thin film with Navier-slip condition (on liquid-solid and liquid-liquid interface); $n=3$: two-phase thin film no-slip condition (on liquid-solid and liquid-liquid interface).}. 
Results regarding the existence of global non-negative weak solutions to a system describing the dynamics of a one-phase thin film with insoluble surfactant are subject in \cite{CT, EW3, GW, Wie}. In \cite{CT} additionally  gravitational forces are included and an upper bound for the non-negative weak solution for the surfactant concentration is stated ($\Gamma \leq 1$). 
 It turns out that the existence of an \emph{energy functional} becomes a crucial part in studying weak solutions of thin films, cf. e.g. \cite{CT, EW3, EW2, EM, GW}, as it provides necessary a-priori estimates, which allow by compactness arguments to extract convergent subsequences of weak solutions to regularized problems tending in the limit to a global weak solution of the original problem.

We impose the following assumptions (similar to \cite{GW}):
Given the surface tension coefficients $\sigma_1=\sigma_1^c$ and $\sigma_2$ of the form
\[\sigma_2(\Gamma)= \sigma_2^c+\sigma(\Gamma),\]
where $\Gamma$ is the surfactant concentration, we assume  the part of the surface tension,  which depends on $\Gamma$, to be non-increasing and the part of the surface tension, which is independent of the concentration of surfactant, to be strictly positive, that is  $\sigma_1^c, \sigma_2^c >0$. We want to emphasize that this in particular implies $R,S>0$.
Moreover, let $\Phi$ be a function, such that:
\begin{itemize}
\item[A1)] $\Phi \in C^2(\R)$ with $\Phi(1)=\Phi^\prime(1)=0$ and
\begin{equation}\label{p}
\Phi^{\prime\prime}(s)=-\frac{\sigma^\prime(s)}{s} \qquad \mbox{for all}\; s\in\R.
\end{equation}
\item[A2)] There exists $c_\Phi>0$ such that $\Phi^{\prime\prime}(s)\geq c_\Phi $ for all $s\in \R$.
\item[A3)] There exists $C_\Phi>0$ and some $r\in (0,1)$ for which $\Phi^{\prime\prime}(s)\leq C_\Phi(|s|^r+1) $ for all $s\in \R$.
\end{itemize}
In A1)--A3), we suppose the assumptions to hold  on the whole real line instead of the physically relevant range $[0,\infty)$. For our purpose, this is needed due to the fact that a-priori it is not clear whether the solution  we construct for the surfactant concentration is non-negative. Unfortunately, theses assumptions do not allow to consider surface tension profiles as commonly used and suggested in e.g. \cite{GG}. In \cite{EW3} the existence of non-negative weak solutions for the one-phase thin film with insoluble surfactant is shown under less restrictive assumptions on the surface profile, which allows for more general surface tension profiles.

 Strongly relying on the approaches in \cite{EW3, GW, Wie}, where global weak solutions to a one-phase thin film model with insoluble surfactant are proved and \cite{EM}, where the existence of global weak solutions to a two-phase thin film model is shown, this contribution combines these results and presents the existence of global weak solutions for the fourth-order two-phase thin film problem with insoluble surfactant \eqref{system2}. Moreover, we make evident that the solutions corresponding to non-negative initial data stay non-negative almost everywhere, which is achieved by similar methods as in \cite{EW3, EM, GW, Wie}. 
 
Let us begin with rewriting \eqref{system2} in a form more convenient for our purpose:
  \begin{align}
  \label{system22}
  \begin{split}
 &\partial_t f +\partial_x \left(\frac{f^{\frac{3}{2}}}{\sqrt{3}}J_f\right)=0,\\[10pt]
  &\partial_t g + \partial_x \left(\frac{\sqrt{3}}{2}g\sqrt{f}J_{f,g}+ \frac{g^{\frac{3}{2}}}{\sqrt{3}}J_g\right)=0, \\[10pt]
 &\partial_t \Gamma +\partial_x \left(\frac{\sqrt{3}}{2}\Gamma \sqrt{f}J_{f,g}+\frac{\sqrt{3}}{2}\Gamma \sqrt{g}J_g+\frac{1}{4}\Gamma g\partial_x \sigma(\Gamma)-D\partial_x\Gamma\right)=0,
 \end{split}
 \end{align}
 where $J_f, J_{f,g}$ and $J_g$ are given by
 \begin{align} 
 \label{Jf}
	&J_f:=\sqrt{f}\left[ \frac{f\partial_x^3((R+S\mu) f+S\mu g)}{\sqrt{3}}+\frac{\sqrt{3}}{2}\mu\left(S g \partial_x^3 (f+g)+\partial_x \sigma(\Gamma)\right)\right], \\[5pt]
 \label{Jfg}
	&J_{f,g}:=\sqrt{f} \left[ \frac{f\partial_x^3((R+S\mu) f+S\mu g)}{\sqrt{3}}+\frac{2}{\sqrt{3}}\mu\left(S g \partial_x^3 (f+g)+\partial_x \sigma(\Gamma)\right)\right],\\[5pt]
 \label{Jg}
	&J_g:=\sqrt{g} \left[\frac{S}{\sqrt{3}}g\partial_x^3 (f+g)+\frac{\sqrt{3 }}{2}\partial_x \sigma(\Gamma)\right].
\end{align}

Given $T\in (0,\infty]$, let $\Omega_T:= (0,T)\times (0,L)$ be the time-space domain. Furthermore, we denote by $\langle \cdot,\cdot \rangle_E$ the dual pairing between the spaces $E^\prime$ and $E$. The main theorem reads as follows:
\begin{thm}[Global Weak Solutions]\label{MT}
Let $f^0,g^0 \in H^1(0,L)$ and $\Gamma^0\in L_{2(r+1)}(0,L)$, where $r\in (0,1)$ corresponds to Assumption A3), be non-negative functions.  Then, there exists at least one global weak solution $(f,g,\Gamma)$ of problem \eqref{system2}--\eqref{wNB} in the sense that for all $T>0$
\begin{itemize}
\item[a)] the solution has the regularity
	\begin{align*}
	&f, g \in L_\infty(0,T;H^1(0,L))\cap C([0,T];C^\alpha([0,L]))\quad \mbox{for all}\quad \alpha \in \Big[0,\frac{1}{2}\Big), \\[5pt]
	&\Gamma \in L_\infty(0,T;L_2(0,L))\cap L_2(0,T;H^1(0,L)),\\[5pt]
	&\partial_t f, \partial_t g, \in L_2(0,T;(H^1(0,L))^\prime)\quad\mbox{and} \quad \partial_t \Gamma \in L_{\frac{3}{2}}(0,T;(W_3^1(0,L))^\prime),
	\end{align*} 
\item[b)] $(f,g,\Gamma)(0)= (f^0,g^0,\Gamma^0)$ and $f\geq 0$, $g\geq 0$, $\Gamma\geq 0$ in $\Omega_T$, where the claims for $\Gamma$ are to be understood as almost everywhere,
\item[c)] the mass of the fluids and the surfactant concentration is conserved, that is
	\begin{equation*}
	\|f(t)\|_{L_1(0,L)} = \|f^0\|_{L_1(0,L)},\quad \|g(t)\|_{L_1(0,L)} = \|g^0\|_{L_1(0,L)}, \quad  \|\Gamma(t)\|_{L_1(0,L)} = \|\Gamma^0\|_{L_1(0,L)}
	\end{equation*}
	for almost all $t\geq 0$,
	
\item[d)] defining the sets $\mathcal{P}_f:=\{(t,x)\in\Omega_T:\;f(t,x)>0\}$and $\mathcal{P}_g:=\{(t,x)\in\Omega_T:\; g(t,x)>0\}$,
	we have $\partial_x^3f,\partial_x^3 g \in L_2(\mathcal{P}_f\cap \mathcal{P}_g)$ and there exist functions $J_f^*,J_{f,g}^*,J_g^*\in L_2(\Omega_T)$, which can be identified on the set $\mathcal{P}_f\cap \mathcal{P}_g$ with $J_f,J_{f,g},J_g$ defined in \eqref{Jf}--\eqref{Jg} so that
	\begin{align}
	\label{MT11} &\displaystyle{\int_0^T} \langle \partial_t f(t),  \xi(t)\rangle_{H^1(0,L)}\,dt=\displaystyle{\int_{\Omega_T}} \left(\frac{f^{\frac{3}{2}}}{\sqrt{3}}J_f^*\right) \partial_x \xi \,d(x,t), \\[5pt]
	\label{MT22} &\displaystyle{\int_0^T} \langle \partial_t g(t),  \xi(t)\rangle_{H^1(0,L)}\,dt=\displaystyle{\int_{\Omega_T}} \left(\frac{\sqrt{3}}{2}g\sqrt{f}J_{f,g}^*+ \frac{g^{\frac{3}{2}}}{\sqrt{3}}J_g^*\right) \partial_x \xi \,d(x,t), \\[5pt]
	\label{MT3}
	\begin{split} &\displaystyle{\int_0^T} \langle \partial_t \Gamma(t),  \xi(t)\rangle_{W^1_3(0,L)}\,dt=\displaystyle{\int_{\Omega_T}} \left(\frac{\sqrt{3}}{2}\Gamma \sqrt{f}J_{f,g}^*+\frac{\sqrt{3}}{2}\Gamma \sqrt{g}J_g^* \right) \partial_x \xi \,d(x,t) \\[5pt]
	&\qquad\qquad\qquad\qquad\qquad\qquad  +\displaystyle{\int_{\Omega_T}} \left(\frac{1}{4}\Gamma g\partial_x \sigma(\Gamma)-D\partial_x\Gamma\right) \partial_x \xi \,d(x,t) 
	\end{split}
	\end{align}
	for all $\xi \in C^\infty(\overline{\Omega}_T)$,
\item[e)] the energy inequality 
\begin{align*}
\mathcal{E}(f,g,\Gamma)(T)+\mathcal{D}(f,g,\Gamma)(T) \leq \mathcal{E}(f^0,g^0,\Gamma^0)
\end{align*}
is satisfied, where
\begin{align*}
\mathcal{E}(f,g,\Gamma)(T):= \displaystyle{\int_0^L}\left\{ \frac{1}{2}\left( R |\partial_xf(T,x)|^2+S\mu|\partial_x(f+g)(T,x)|^2\right)+\mu\Phi(\Gamma(T,x))\right\}\,dx
\end{align*}
and
\begin{align*}
\mathcal{D}(f,g,\Gamma)(T) :=& -\int_{\mathcal{P}_f\cap\mathcal{P}_g} \left\{ f|J_{f}|^2 +g\mu|J_g|^2  +\frac{f \mu^2 }{4}\left[S g \partial_x^3 (f+g)+\partial_x \sigma(\Gamma) \right]^2 \right.\\[5pt]
&\qquad\qquad\quad  +\frac{ g \mu}{4}|\partial_x\sigma(\Gamma)|^2+\mu \Phi^{\prime\prime}(\Gamma)D|\partial_x \Gamma|^2 \Bigg\}\,d(x,t).
 \end{align*}
\end{itemize}
\end{thm}
Owing to the degeneracy of the system, proving the existence of non-negative global weak solutions to \eqref{system2}, requires a two-step compactness method. In accordance to \cite{EW3, EM, GW, Wie}, we construct first a family of suitably regularized, non-degenerate systems and prove by using Galerkin approximations, a-priori estimates and compactness arguments that there exist global weak solutions to the regularized problems (Section \ref{TRS}). In a second step we show that a sequence of weak solutions to the regularized problems tends in the limit to a non-negative weak solution of the original problem (Section \ref{EGWS}).

 \section{The Regularized Systems} \label{TRS}
 We define for every $\e \in (0,1]$ the function $a_\e : \R \longrightarrow \R_+$ by
\begin{equation*}
a_\e(s):= \e + \max\{0,s\}.
\end{equation*}
Furthermore, we introduce the function
\begin{equation*}
\mathcal{T}(s):= \left\{
\begin{array}{lcl}
s,&\mbox{if}&\; s\in (0,1), \\[5pt]
2-s,\quad &\mbox{if}&\; s\in [1,2], \\[5pt]
0, &\mbox{if}&\; s\geq 2,
\end{array} \right.
\qquad\qquad\qquad \mathcal{T}(s)= \mathcal{T}(-s),\quad\mbox{if}\quad s<0.
\end{equation*}
We put $\mathcal{T}_\e:= \ei\mathcal{T}(\cdot \,\e)$ for $\e \in (0,1]$
and set
\begin{equation*}
\sigma_\e(s):= \int_1^s \mathcal{T}_\e(\sigma^\prime(\tau))\, d\tau\qquad \mbox{for}\quad s\in \R.
\end{equation*}
Note that by construction  and Assumption S2), we find that $\sigma_\e\in C^{1,1}(\R)$ and
\begin{equation}\label{swell}
|\sigma_\e^\prime(s)|\leq |\sigma^\prime(s)| \qquad\mbox{for all}\quad  s\in \R.
\end{equation} 
Associated to $\sigma_\e$, we introduce a truncation of the identity
\begin{equation}\label{taudef}
\tau_\e(s):= s\frac{\sigma_\e^\prime(s)}{\sigma^\prime(s)}\qquad \mbox{for}\quad s\in \R.
\end{equation}
This is well-defined in view of \eqref{swell}. We emphasize that $\tau_\e$ is locally Lipschitz having compact support within $C_{\Phi}[-2\ei,2\ei]$ and 
\begin{equation}\label{abtau}
|\tau_\e(s)|\leq |s|\qquad \mbox{for}\quad s\in \R.
\end{equation} 
We introduce the regularized problem:
\begin{align}
\nonumber
&\partial_t f_\e +\partial_x \left[ a_\e(f_\e)\left(\displaystyle{\frac{Ra_\e(f_\e)^2}{3}}\partial_x^3 f_\e + S\mu \left(\displaystyle{\frac{a_\e(f_\e)^2}{3}} +\displaystyle{\frac{a_\e(f_\e)a_\e(g_\e)}{2}} \right) \partial_x^3(f_\e+g_\e) +\mu 	\displaystyle{\frac{a_\e(f_\e)}{2}}\partial_x\sigma_\e(\Gamma_\e)\right)\right]=0,\\[15pt]
\label{rsystem}
\begin{split}
& \partial_t g_\e + \partial_x \left[a_\e(g_\e)\left( \displaystyle{\frac{Ra_\e(f_\e)^2}{2}} \partial_x^3 f_\e +S\left(  \displaystyle{\frac{a_\e(g_\e)^2}{3}} +\mu \left( \displaystyle{\frac{a_\e(f_\e)^2}{2}}+a_\e(f_\e)a_\e(g_\e)\right)\right)\partial_x^3 (f_\e+g_\e) \right.\right.\\[15pt]
&
\left.\left.\qquad\qquad\quad +\left(\mu a_\e(f_\e) +\displaystyle{\frac{a_\e(g_\e)}{2}} \right) \partial_x \sigma_\e(\Gamma_\e)\right)\right]=0, 
\end{split}\\[15pt]
\nonumber
&\partial_t \Gamma_\e +\partial_x\left[\tau_\e(\Gamma_\e)\left( \displaystyle{\frac{Ra_\e(f_\e)^2}{2}} \partial_x^3 f_\e +S\left(  \displaystyle{\frac{a_\e(g_\e)^2}{2}}+\mu \left(\displaystyle{\frac{a_\e(f_\e)^2}{2}}+a_\e(f_\e)a_\e(g_\e)\right)\right)\partial_x^3 (f_\e+g_\e)\right.\right.\\[15pt]
\nonumber
&\left.\left.\qquad\qquad\quad+\left(\mu a_\e(f_\e)+a_\e(g_\e)\right) \partial_x \sigma_\e(\Gamma_\e)\right)- D\partial_x\Gamma_\e\right]=0
\end{align}
in $ \Omega_T $ supplemented by the initial and boundary conditions \eqref{wID}, \eqref{wNB}. The function $a_\e$ yields the regularizing effect that the system \eqref{rsystem} is no longer degenerate, but uniformly parabolic \footnote{the coefficients of the fourth-order terms in the equations for $f_\e$ and $g_\e$ are bounded from below by $\e>0$}, whereas the replacement by the truncation function $\tau_\e$  will be needed for proving the non-negativity of a weak solution $\Gamma$.
In accordance to \eqref{system22} we can rewrite the system above in a more compact form as
\begin{align} 
\nonumber
&\partial_t f_\e +\partial_x \left( \frac{1}{\sqrt{3}}a_\e(f_\e)^\frac{3}{2} J_{f}^\e\right) =0,\\[10pt]
\label{rrsystem}
& \partial_t g_\e + \partial_x \left(\frac{\sqrt{3}}{2}a_\e(g_\e)\sqrt{a_\e(f_\e)}J_{f,g}^\e+ \frac{1}{\sqrt{3}}a_\e(g_\e)^{\frac{3}{2}}J_{g}^\e\right)=0, \\[10pt]
\nonumber
& \partial_t \Gamma_\e +\partial_x \left(\frac{\sqrt{3}}{2}\tau_\e(\Gamma_\e) \sqrt{a_\e(f_\e)}J_{f,g}^\e+\frac{\sqrt{3}}{2}\tau_\e(\Gamma_\e)\sqrt{a_\e(g_\e)}J_{g}^\e+\frac{1}{4}\tau_\e(\Gamma_\e) a_\e(g_\e) \partial_x \sigma_\e(\Gamma_\e)-D\partial_x\Gamma_\e\right)=0,
\end{align}
 where 
 \begin{align*} 
	&J_{f}^\e:=\sqrt{a_\e(f_\e)}\left[ \frac{a_\e(f_\e)}{\sqrt{3}}\partial_x^3((R+S\mu) f_\e+S\mu g_\e)+\frac{\sqrt{3}}{2}\mu\left(S a_\e(g_\e) \partial_x^3 (f_\e+g_\e)+\partial_x \sigma_\e(\Gamma_\e)\right)\right], \\[5pt]
	&J_{f,g}^\e:=\sqrt{a_\e(f_\e)} \left[ \frac{a_\e(f_\e)}{\sqrt{3}}\partial_x^3((R+S\mu) f_\e+S\mu g_\e)+\frac{2}{\sqrt{3}}\mu\left(S a_\e(g_\e) \partial_x^3 (f_\e+g_\e)+\partial_x \sigma_\e(\Gamma_\e)\right)\right],\\[5pt]
	&J_{g}^\e:=\sqrt{a_\e(g_\e)} \left[\frac{S}{\sqrt{3}}a_\e(g_\e)\partial_x^3 (f_\e+g_\e)+\frac{\sqrt{3 }}{2}\partial_x \sigma_\e(\Gamma_\e)\right].
\end{align*}
We show that for any fixed $\e\in (0,1]$ the problem \eqref{rsystem}, supplemented by the initial and boundary conditions \eqref{wID}, \eqref{wNB}, admits a global weak solution.

\begin{thm}[Global Weak Solutions for the Regularized Systems]\label{T31}
Let $\e \in (0,1]$ be fixed and $(f^0,g^0,\Gamma^0) \in  (H^1(0,L))^2 \times L_{2(r+1)}(0,L)$, where $r\in (0,1)$ corresponds to Assumption A3). Then, for any $T>0$ there exists at least one triple of functions $(f_\e,g_\e, \Gamma_\e)$ having the regularity
\begin{align*}
&f_\e, g_\e \in L_\infty(0,T;H^1(0,L))\cap L_2(0,T;H^3(0,L))\cap C([0,T];C^\alpha([0,L])), \quad \alpha\in \Big[0,\frac{1}{2}\Big), \\[5pt]
&\Gamma_\e \in L_\infty(0,T;L_2(0,L))\cap L_2(0,T;H^1(0,L)), \\
&\partial_t f_\e, \partial_t g_\e, \partial_t \Gamma_\e\in L_2(0,T;(H^1(0,L))^\prime),
\end{align*}
satisfying
\begin{align}\label{T1}
&\displaystyle{\int_0^T}  \langle \partial_t f_\e(t), \xi(t) \rangle_{H^1(0,L)} \,dt= \displaystyle{\int_{\Omega_T}}  \left(\frac{a_\e(f_\e)^{\frac{3}{2}}}{\sqrt{3}}J_f^\e\right)\partial_x \xi \,d(x,t)  ,\\[5pt]
\label{T2}
&\displaystyle{\int_0^T}  \langle \partial_t g_\e(t), \xi(t) \rangle_{H^1(0,L)} \,dt  = \displaystyle{\int_{\Omega_T}} \left(\frac{\sqrt{3}}{2}a_\e(g_\e)\sqrt{a_\e(f_\e)}J_{f,g}^\e+ \frac{a_\e(g_\e)^{\frac{3}{2}}}{\sqrt{3}}J_g^\e\right) \partial_x \xi \,d(x,t),\\[5pt]
\label{T3}
\begin{split}
&\displaystyle{\int_0^T}  \langle \partial_t \Gamma_\e(t), \xi(t) \rangle_{W^3_1(0,L)} \,dt = \int_{\Omega_T}  \left(\frac{\sqrt{3}}{2}\tau_\e(\Gamma_\e) \sqrt{a_\e(f_\e)}J_{f,g}^\e\right) \partial_x \xi  \,d(x,t)\\[5pt]
&\qquad\qquad\qquad+\int_{\Omega_T}\left(\frac{\sqrt{3}}{2}\tau_\e(\Gamma_\e) \sqrt{a_\e(g_\e)}J_{g}^\e+\frac{1}{4}\tau_\e(\Gamma_\e) a_\e(g_\e) \partial_x \sigma_\e(\Gamma_\e)-D\partial_x\Gamma_\e\right) \partial_x \xi  \,d(x,t),
  \end{split}
\end{align}
for all $\xi \in C^\infty(\overline{\Omega}_T)$.
Furthermore\footnote{where $\Gamma_\e(0,\cdot)=\Gamma_0$ holds almost everywhere.},
\begin{align}\label{inC}
(f_\e(0,\cdot), g_\e(0,\cdot), \Gamma_\e(0,\cdot))=(f^0,g^0,\Gamma^0)
\end{align}
and the mass of the fluids and the surfactant concentration is preserved
\begin{equation}\label{mass}
\displaystyle{\int_0^L}f_\e(t)\, dx = \|f^0\|_{L_1(0,L)},\qquad \displaystyle{\int_0^L}g_\e(t)\, dx = \|g^0\|_{L_1(0,L)},\qquad \displaystyle{\int_0^L}\Gamma_\e(t)\, dx = \|\Gamma^0\|_{L_1(0,L)}
\end{equation}
for almost all $t\geq 0$.
Moreover, there holds the energy inequality
\begin{align}\label{energy}
\mathcal{E}(f_\e,g_\e,\Gamma_\e)(T) + \mathcal{D}_\e(f_\e,g_\e,\Gamma_\e)(T) \leq \mathcal{E}(f^0,g^0,\Gamma^0)
\end{align}
for almost all $T\geq 0$, where
\begin{align*}
 \mathcal{D}_\e(f_\e,g_\e,\Gamma_\e)(T):= &-\int_{\Omega_T} \left\{ |J_{f}^\e|^2 +\mu |J_g^\e|^2+\frac{ a_\e(f_\e)\mu^2}{4}\left[S a_\e(g_\e) \partial_x^3 (f_\e+g_\e)-\partial_x \sigma_\e(\Gamma_\e) \right]^2\right.\\[5pt]
&\qquad  +\frac{ a_\e(g_\e) \mu}{4}|\partial_x\sigma_\e(\Gamma_\e)|^2+\mu \Phi^{\prime\prime}(\Gamma_\e)D|\partial_x \Gamma_\e|^2 \Bigg\}\,d(x,t).
 \end{align*}
\end{thm}

\subsection{Approximation of a Weak Solution  by Fourier Series Expansions}
 Let $\e\in (0,1]$ be fixed. Following \cite{EM, GW, Wie}, we construct a solution to \eqref{rsystem}, \eqref{wID} and \eqref{wNB} by the method of Galerkin approximations. That is, we are seeking for functions $f^n_\e,g^n_\e,\Gamma^n_\e$, such that the problem is satisfied in a weak sense, when testing against functions from an $n$--dimensional subspace. These solutions are called \emph{Galerkin approximations}. 

Note that the normalized eigenvectors of $-\Delta : H^2(0,L)\longrightarrow L_2(0,L)$, which satisfy zero Neumann--boundary conditions are given by
\[\phi_0:= \sqrt{\displaystyle{\frac{1}{L}}}\qquad\qquad \mbox{and} \qquad \qquad \phi_k:= \sqrt{\displaystyle{\frac{2}{L}}}\cos \left( \displaystyle{\frac{k\pi x}{L}}\right),\; k\geq 1,\]
and form an orthonormal basis in $L_2(0,L)$. It is known that any function $f$ belonging to $H^1(0,L)$ can be written as $\sum_{k=0}^\infty \alpha_k \phi_k$, where the series converges in $H^1(0,L)$ and $\alpha_k:= \left( f\mid \phi_k\right)_2$ for $k\geq 0$ with $(\cdot\mid \cdot)_2$ being the scalar product in $L_2(0,L)$. 
We take a Galerkin ansatz for $f_\e,g_\e$ and $\Phi^\prime(\Gamma_\e)$. In view of Assumption A1), A2), there exists a continuous differentiable inverse function $W:=(\Phi^\prime)^{-1}$. Set $v_\e:= \Phi^\prime (\Gamma_\e)$, then $\Gamma_\e=W(v_\e)$ and
the regularized system \eqref{rsystem} becomes
\begin{align*} 
&\partial_t f_\e +\partial_x \left[ a_\e(f_\e)\left(\displaystyle{\frac{Ra_\e(f_\e)^2}{3}}\partial_x^3 f_\e + S\mu \left(\displaystyle{\frac{a_\e(f_\e)^2}{3}} +\displaystyle{\frac{a_\e(f_\e)a_\e(g_\e)}{2}} \right) \partial_x^3(f_\e+g_\e)\right.\right.\\[5pt]
&\left.\left.\qquad\qquad\qquad\qquad-\mu 	\displaystyle{\frac{a_\e(f_\e)}{2}}\tau_\e(W(v_\e))\partial_x v_\e\right)\right]=0,\\[5pt]
& \partial_t g_\e + \partial_x \left[a_\e(g_\e)\left( \displaystyle{\frac{Ra_\e(f_\e)^2}{2}} \partial_x^3 f_\e +S\left(  \displaystyle{\frac{a_\e(g_\e)^2}{3}} +\mu \left( \displaystyle{\frac{a_\e(f_\e)^2}{2}}+a_\e(f_\e)a_\e(g_\e)\right)\right)\partial_x^3 (f_\e+g_\e) \right.\right.\\[5pt]
&\left.\left.\qquad\qquad\qquad\qquad -\left(\mu a_\e(f_\e)+ \displaystyle{\frac{a_\e(g_\e)}{2}} \right) \tau_\e(W(v_\e))\partial_x v_\e\right)\right]=0, \\[5pt]
&\partial_t W(v_\e)+\partial_x\left[\tau_\e(W(v_\e))\left( \displaystyle{\frac{R a_\e(f_\e)^2}{2}} \partial_x^3 f_\e +\left( S \displaystyle{\frac{a_\e(g_\e)^2}{2}}+S\mu\left(\displaystyle{\frac{a_\e(f_\e)^2}{2}}+a_\e(f_\e)a_\e(g_\e)\right)\right)\partial_x^3(f_\e+g_\e)\right. \right.\\[5pt]
&\qquad\qquad\qquad\qquad - \left(\mu a_\e(f_\e)+a_\e(g_\e)\right)\tau_\e(W(v_\e)) \partial_x v_\e\Big)- D\partial_x W(v_\e)\Big]=0,
\end{align*}
in view of $\partial_x \sigma_\e(W(v))=-\tau_\e(W(v))\partial_x v$.
Observe that Assumption A3) and $\Gamma^0\in L_{2(r+1)}(0,L)$ imply that 
$\Phi^\prime(\Gamma^0)\in L_2(0,L)$. 
For $f^0, g^0 \in H^1(0,L)$ and $v^0:=\Phi^\prime(\Gamma^0)\in L_2(0,L)$  there exist 
sequences $(f_{0k})_{k\in \N},(g_{0k})_{k\in \N}$ and $(v_{0k})_{k\in \N}$, such that
\begin{align*}
f_{0}^n:=\displaystyle{\sum_{k=0}^nf_{0k}\phi_k} \qquad \mbox{with}\qquad f_{0}^n \longrightarrow f^0 \quad \mbox{in}\; H^1(0,L),\\
g_{0}^n:=\displaystyle{\sum_{k=0}^ng_{0k}\phi_k} \qquad \mbox{with}\qquad g_{0}^n \longrightarrow g^0 \quad \mbox{in}\; H^1(0,L),\\
v_{0}^n:=\displaystyle{\sum_{k=0}^nv_{0k}\phi_k} \qquad \mbox{with}\qquad v_{0}^n \longrightarrow v^0 \quad \mbox{in}\; L_2(0,L).
\end{align*}
We seek for continuously differentiable functions with respect to time
\[f_\e^n(t,x):= \displaystyle{\sum_{k=0}^n}F_\e^k(t)\phi_k(x),\quad g_\e^n(t,x):= \displaystyle{\sum_{k=0}^n}G_\e^k(t)\phi_k(x),\quad v_\e^n(t,x):= \displaystyle{\sum_{k=0}^n}V_\e^k(t)\phi_k(x)\quad\mbox{in}\quad \Omega_T,\]
which solve \eqref{rsystem} when testing with functions from the linear subspace spanned by $\{\phi_0,\ldots,\phi_n\}$ and satisfy initially
\[f_\e^n(0,\cdot)=f_0^n,\qquad g_\e^n(0,\cdot)=g_0^n,\qquad v_\e^n(0,\cdot)=v_0^n.\]
Set $\Gamma_\e^n:= W(v_\e^n)$.
By construction the functions $f_\e^n,g_\e^n, \Gamma_\e^n$ satisfy the boundary condition \eqref{wNB}. 
\begin{lem}\label{gan} Let $\e\in (0,1]$ be fixed and $T>0$. Then, the problem \eqref{rsystem}, \eqref{wID}, \eqref{wNB} admits for every $n\in\N$ a unique global Galerkin approximation $(f_\e^n,g_\e^n,\Gamma^n_\e)$.
Furthermore, conservation of mass
\begin{equation}\label{ccc}
\displaystyle{\int_0^L}f_\e^n(t)\, dx = \|f^0\|_1,\qquad \displaystyle{\int_0^L}g_\e^n(t)\, dx = \|g^0\|_1,\qquad \displaystyle{\int_0^L}\Gamma_\e^n(t)\, dx = \|\Gamma^0\|_1
\end{equation}
holds true 
for all $t\geq 0$ and the energy equality
\begin{align}\label{energy1}
\mathcal{E}(f_\e^n,g_\e^n,\Gamma_\e^n)(T) + \mathcal{D}_\e(f_\e^n,g_\e^n,\Gamma_\e^n)(T)=  \mathcal{E}(f_0^n,g_0^n,\Gamma_0^n)
\end{align}
is satisfied.
\end{lem}

\begin{proof} We test the equations in \eqref{rsystem} successively with $\phi_0,\ldots ,\phi_n$ and integrate by parts. Due to the boundary conditions and the special structure of the equations in \eqref{rsystem}, the boundary terms vanish and we obtain a system of ordinary differential equations, which can be solved locally by the Picard--Lindelöf Theorem. Testing \eqref{rsystem} against $\phi_j$ for some $j\in \{0,\ldots, n\}$ yields
\begin{align}
\label{F}
\begin{split}
&\left(\partial_t f_\e^n \mid \phi_j\right)_2 = \left(a_\e(f_\e^n)\left(  R \displaystyle{\frac{a_\e(f_\e^n)^2}{3}} \partial_x^3 f_\e^n + S\mu \left(\displaystyle{\frac{a_\e(f_\e^n)^2}{3}} +\displaystyle{\frac{a_\e(f_\e^n)a_\e(g_\e^n)}{2}} \right)\partial_x^3 (f_\e^n+g_\e^n) \right.\right. \\[5pt] 
&\left.\hspace{2cm} \qquad-\mu \displaystyle{\frac{a_\e(f_\e^n)}{2}}\tau_\e(W(v_\e^n))\partial_x v_\e^n \Big) \,\Big|\, \partial_x \phi_j \right)_2, 
\end{split}\\[10pt]
\label{G}
\begin{split}
&\left(\partial_t g_\e^n \mid\phi_j\right)_2 = \left(a_\e(g_\e^n)\left( R\displaystyle{\frac{a_\e(f_\e^n)^2}{2}}\partial_x^3 f_\e^n +\left( S \displaystyle{\frac{a_\e(g_\e^n)^2}{3}} +S\mu \left( \displaystyle{\frac{a_\e(f_\e^n)^2}{2}}+a_\e(f_\e^n)a_\e(g_\e^n)\right)\right) \right. \right.\\[5pt]
 &\hspace{2cm}\left.\qquad \left. \times \partial_x^3 (f_\e^n+g_\e^n)+\left(\mu a_\e(f_\e^n) -\displaystyle{\frac{a_\e(g_\e^n)}{2}} \right) \tau_\e(W(v_\e^n))\partial_x v_\e^n\right)\,\Big|\, \partial_x \phi_j \right)_2,
 \end{split}\\[10pt]
 \label{Gamma}
 \begin{split}
 &\left(\partial_t W(v^n_\e ) \mid \phi_j\right)_2 = \left(\tau_\e(W(v^n_\e))\left( R \displaystyle{\frac{a_\e(f_\e^n)^2}{2}} \partial_x^3 f_\e^n +\left( S \displaystyle{\frac{a_\e(g_\e^n)^2}{2}}+S\mu\left(\displaystyle{\frac{a_\e(f_\e^n)^2}{2}}+a_\e(f_\e^n)a_\e(g_\e^n)\right)\right) \right. \right.\\[5pt]
  &\hspace{1cm}\times \partial_x^3 (f_\e^n+g_\e^n)- \left(\mu a_\e(f_\e^n)+a_\e(g_\e^n)\right)\tau_\e(W(v_\e)) \partial_x v^n_\e\Big)- D\partial_x W(v^n_\e) \,\Big|\, \partial_x \phi_j \Big)_2.
  \end{split}
\end{align}
Define $\Psi:=(\Psi_1,\Psi_2,\Psi_3):\R^{3(n+1)}\longrightarrow \R^{3(n+1)}$ by
\begin{align*}
&\Psi_{1,j}(p,q,r):=  \displaystyle{\sum_{k=1}^n p^k}\left( R \displaystyle{\frac{a_\e(\Theta_f(p))^3}{3}}\partial_x^3 \phi_k \Big| \partial_x \phi_j \right)_2 \\[10pt]
 &\quad+ \displaystyle{\sum_{k=1}^n(p^k+q^k )}\left( S\mu \left(\displaystyle{\frac{a_\e(\Theta_f(p))^3}{3}} +\displaystyle{\frac{a_\e(\Theta_f(p))^2a_\e(\Theta_g(q))}{2}} \right) \partial_x^3 \phi_k \,\Big|\, \partial_x \phi_j \right)_2 \\[10pt]
&\quad-\displaystyle{\sum_{k=1}^nr^k}\left( \mu \displaystyle{\frac{a_\e(\Theta_f(p))^2}{2}}\tau_\e(W(\Theta_v(r)))\partial_x \phi_k \,\Big|\, \partial_x \phi_j \right)_2,\\[10pt]
&\Psi_{2,j}(p,q,r):= \displaystyle{\sum_{k=1}^n p^k}\left( R\displaystyle{\frac{a_\e(\Theta_f(p))^2a_\e(\Theta_g(q))}{2}}\partial_x^3 \phi_k \,\Big|\, \partial_x \phi_j \right)_2+\displaystyle{\sum_{k=1}^n (p^k+q^k)}\left( \left(\displaystyle{\frac{ S a_\e(\Theta_g(q))^3}{3}}\right.\right.  \\[10pt]
&\qquad\qquad\qquad\left.\left. +S\mu \left( \displaystyle{\frac{a_\e(\Theta_f(p))^2a_\e(\Theta_g(q))}{2}}+a_\e(\Theta_f(p))a_\e(\Theta_g(q))^2\right)\right)\partial_x^3 \phi_k \,\Big|\, \partial_x \phi_j \right)_2\\[10pt]
&\quad-\displaystyle{\sum_{k=1}^n r^k}\left( \left(\mu a_\e(\Theta_f(p))a_\e(\Theta_g(q)) +\displaystyle{\frac{a_\e(\Theta_g(q))^2}{2}} \right)\tau_\e(W(\Theta_v(r))) \partial_x \phi_k \,\Big|\, \partial_x \phi_j \right)_2
\end{align*}
and
\begin{align*}
&\Psi_{3,j}(p,q,r):= \displaystyle{\sum_{k=1}^n p^k}\left( R \displaystyle{\frac{a_\e(\Theta_f(p))^2}{2}}W(\Theta_v(r)) \partial_x^3 \phi_k \,\Big|\, \partial_x \phi_j \right)_2+ \displaystyle{\sum_{k=1}^n (p^k+q^k)}\left(\left( \displaystyle{\frac{ Sa_\e(\Theta_f(p))^2}{2}}\right.\right.\\[10pt]
&\qquad\qquad\qquad\left.\left.+S\left(\displaystyle{\frac{a_\e(\Theta_f(q))^2}{2}}+a_\e(\Theta_f(p))a_\e(\Theta_g(q))\right)\right) W(\Theta_v(r)) \partial_x^3 \phi_k\,\Big|\,\partial_x \phi_j \right)_2\\[10pt]
&\quad- \displaystyle{\sum_{k=1}^n r^k}\left( \left(\mu a_\e(\Theta_f(p))+a_\e(\Theta_g(q))\right)(\tau_\e(W(\Theta_v(r))))^2 \partial_x \phi_k+ DW^\prime(\Theta_v(r)) \partial_x \phi_k \,\Big|\, \partial_x \phi_j \right)_2,
\end{align*}
for $j\in\{0,\ldots,n\}$,  $p=(p^0,\ldots,p^n)$, $q=(q^0,\ldots,q^n)$ and $r=(r^0,\ldots, r^n)$ being elements in $\R^n$, and
\[\Theta_f(p):=  \displaystyle{\sum_{k=0}^n}p^k\phi_k,\qquad \Theta_g(q):=   \displaystyle{\sum_{k=0}^n}q^k\phi_k,\qquad \Theta_v(r):= \displaystyle{\sum_{k=0}^n}r^k\phi_k.\]
For $(F,G,V):=(F_\e^0,\ldots, F_\e^n, G_\e^0,\ldots, G_\e^n,V^0_\e,\ldots, V^n_\e)$ the function $\Psi(F,G,V)$ represents the right-hand side of \eqref{F}--\eqref{Gamma}.
Note that the left-hand side of \eqref{F} satisfies
\[\left(\partial_t f_\e^n\mid \phi_j\right)_2 = \left(\partial_t \Theta_f(F)\mid \phi_j\right)_2 = \partial_t F_\e^j\quad \mbox{for all}\quad j\in\{0,\ldots,n\}.\]
The analog relation holds true for the left-hand sides in \eqref{G} and \eqref{Gamma}. We obtain the ordinary differential equation
\begin{equation}\label{ode}
\displaystyle{\frac{d}{dt}}(F,G,V) = \Psi(F,G,V),\qquad (F,G,V)(0)=(f_{00},\ldots f_{0n}, g_{00},\ldots g_{0n}, v_{00},\ldots v_{0n}).
\end{equation}
The function
$\Psi=(\Psi_1,\Psi_2,\Psi_3):\R^{3(n+1)}\rightarrow \R^{3(n+1)}$ is locally Lipschitz continuous, for $a_\e$ as well as $\tau_\e$ have this property. Thus, problem \eqref{ode} admits a unique local solution $(F,G,V)\in \left( C^1([0,T_\e^n),\R^n)\right)^3$, where $[0,T_\e^n)$ is the maximal time interval of existence.\footnote{We deduce in particular that $F_\e^0 = f_{00},G_\e^0 = g_{00},B_\e^0 = g_{00} $ are independent of time since $\Psi_{i,0}=0$ for $i=1,2,3$.} Hence,
\begin{align*}
f_\e^n, g_\e^n \in  C^1([0,T_\e^n);C^\infty([0,L])),\qquad
\Gamma_\e^n \in  C^1([0,T_\e^n);C^1([0,L])) 
\end{align*}
 is a local Galerkin approximation of \eqref{rsystem}. 
In order to prove that the solution is global in time for every $n\in\N$, we use that the functional
\[\mathcal{E}(f_\e^n, g_\e^n , \Gamma_\e^n)= \displaystyle{\int_0^L} \left\{ \displaystyle{\frac{1}{2}}\left(R |\partial_x f_\e^n |^2+ S \mu |\partial_x (f_{\e}^n+g_\e^n) |^2\right)+\mu \Phi(\Gamma^n_\e)\right\}\,dx\]
decreases along the solution $(f_\e^n,g_\e^n,\Gamma_\e^n)$ of \eqref{rsystem}. The time derivative of $\mathcal{E}(f_\e^n, g_\e^n , \Gamma_\e^n)$ yields
\begin{align*}
\displaystyle{\frac{d}{dt}} \mathcal{E}(f_\e^n, g_\e^n , \Gamma^n_\e)&= \displaystyle{\frac{d}{dt}}\displaystyle{\int_0^L} \left\{ \displaystyle{\frac{1}{2}}\left(R |\partial_x f_\e^n |^2+ S \mu |\partial_x (f_{\e}^n+g_\e^n) |^2\right)+\mu \Phi(\Gamma^n_\e)\right\}\,dx \\[5pt]
&= \displaystyle{\int_0^L} \left\{ R \partial_x f_\e^n \partial_x\partial_t f_\e^n +S\mu\partial_x(f_\e^n+g_\e^n)\partial_x\partial_t(f_\e^n+g_\e^n) + \mu \Phi^\prime(\Gamma_\e^n)\partial_t \Gamma_\e^n\right\}\,dx \\[5pt]
&= - \displaystyle{\int_0^L}\left\{ R \partial_x^2 f_\e^n\partial_t f_\e^n +S\mu \partial_x^2(f_\e^n+g_\e^n)\partial_t(f_\e^n+g_\e^n)  - \mu v^n_\e\partial_t \Gamma_\e^n\right\}\,dx. 
\end{align*}
Since $\partial_x^2 f_\e^n(t)$, $\partial_x g_\e^n(t)$ as well as  $ v^n_\e(t)= \Phi^\prime(\Gamma_\e^n(t))$ belong to $ \spa\{\phi_0,\ldots,\phi_n\}$ for all $t\in [0,T_\e^n)$, we use them as test functions for the equations in \eqref{rsystem} and obtain that

\allowdisplaybreaks
\begin{align*}
&\displaystyle{\frac{d}{dt}}\displaystyle{\int_0^L} \left\{ \displaystyle{\frac{1}{2}}\left(R |\partial_x f_\e^n |^2+ S \mu |\partial_x (f_{\e}^n+g_\e^n) |^2\right)+\mu \Phi(\Gamma^n_\e)\right\}\,dx \\[5pt]
&\qquad=-\displaystyle{\int_0^L} \left\{ R \partial_x^3 f_\e^n \left[ R \displaystyle{\frac{a_\e(f_\e^n)^3}{3}} \partial_x^3 f_\e^n + S_\mu \left(\displaystyle{\frac{a_\e(f_\e^n)^3}{3}} +\displaystyle{\frac{a_\e(f_\e^n)^2a_\e(g_\e^n)}{2}} \right) \partial_x^3(f_\e^n+g_\e^n) \right.\right. \\[5pt]
& \left.\left. \qquad\qquad\qquad + \mu \displaystyle{\frac{a_\e(f_\e^n)^2}{2}}\partial_x \sigma_\e(W(v_\e^n))\right] \right\}\,dx \\[5pt]
& \qquad\quad-\displaystyle{\int_0^L} \left\{ S\mu \partial_x^3 (f_\e^n+g_\e^n) \left[ R \displaystyle{\frac{a_\e(f_\e^n)^3}{3}} \partial_x^3 f_\e^n + S\mu \left(\displaystyle{\frac{a_\e(f_\e^n)^3}{3}} +\displaystyle{\frac{a_\e(f_\e^n)^2a_\e(g_\e^n)}{2}} \right) \partial_x^3(f_\e^n+g_\e^n) \right.\right. \\[5pt]
& \qquad\qquad\qquad + \mu \displaystyle{\frac{a_\e(f_\e^n)^2}{2}}\partial_x \sigma_\e(W(v_\e^n)) +R\displaystyle{\frac{a_\e(f_\e^n)^2a_\e(g_\e^n)}{2}}\partial_x^3 f_\e^n \\[5pt]
& \qquad\qquad\qquad   +\left( S \displaystyle{\frac{a_\e(g_\e^n)^3}{3}} +S\mu \left( \displaystyle{\frac{a_\e(f_\e^n)^2a_\e(g_\e^n)}{2}}+a_\e(f_\e^n)a_\e(g_\e^n)^2\right)\right) \partial_x^3 (f_\e^n+g_\e^n) \\[5pt]
& \left.\left.\qquad\qquad\qquad +\left(\mu a_\e(f_\e^n)a_\e(g_\e^n) +\displaystyle{\frac{a_\e(g_\e^n)^2}{2}} \right) \partial_x \sigma_\e(W(v_\e^n)) \right]\right\}\, dx \\[5pt] 
& \qquad\quad-\displaystyle{\int_0^L}\left\{ \mu  \left[\left(  S\displaystyle{\frac{a_\e(g_\e^n)^2}{2}}+\mu\left(\displaystyle{\frac{a_\e(f_\e^n)^2}{2}}+a_\e(f_\e^n)a_\e(g_\e^n)\right)\right)\partial_x \sigma_\e(W(v_\e^n)) \partial_x^3 (f_\e^n+g_\e^n)\right. \right.\\[5pt]
&\qquad\qquad\qquad +R \displaystyle{\frac{a_\e(f_\e^n)^2}{2}}\partial_x \sigma_\e(W(v_\e^n)) \partial_x^3 f_\e^n+ \left(\mu a_\e(f_\e^n)+a_\e(g_\e^n)\right)|\partial_x \sigma_\e(W(v_\e^n))|^2 \\[5pt]
& \left.\left. \qquad\qquad\qquad\qquad\qquad +D \Phi^{\prime\prime}(\Gamma_\e^n)|\partial_x\Gamma_\e^n|^2 \right] \right\} \,dx,
\end{align*} 
where we used the relation $\partial_x v_\e^n \partial_x \Gamma_\e^n=\Phi^{\prime\prime}(\Gamma_\e^n)|\partial_x\Gamma_\e^n|^2 $  and that Assumption A1) implies 
\[\partial_xv_\e^n \tau_\e(\Gamma_\e^n)=\Phi^{\prime\prime}(\Gamma_\e^n)\partial_x\Gamma_\e^n\Gamma \frac{\sigma_\e^\prime(\Gamma_\e^n)}{\sigma^\prime(\Gamma_\e^n)}= -\partial_x\sigma_\e^\prime(\Gamma_\e^n) \]
in the last integral above.
After a tedious but  straight forward computation we arrive at
\begin{align*}
&\displaystyle{\frac{d}{dt}}\mathcal{E}(f_\e^n,g_\e^n,\Gamma_\e^n)=-\displaystyle{\int_0^L} \left\{ \frac{1}{4}\mu^2 a_\e(f_\e^n)\left[S a_\e(g_\e^n) \partial_x^3 (f_\e^n+g_\e^n)+\partial_x \sigma_\e(\Gamma_\e^n) \right]^2 \right.\\[5pt]
&\quad\quad a_\e(f_\e^n)\left[ \frac{a_\e(f_\e^n)\partial_x^3((R+S\mu) f_\e^n+S\mu g_\e^n)}{\sqrt{3}}+\frac{\sqrt{3}}{2}\mu\left(S a_\e(g_\e^n) \partial_x^3 (f_\e^n+g_\e^n)+\partial_x \sigma_\e(\Gamma_\e^n)\right)\right]^2  \\[5pt]
&\quad\quad+a_\e(g_\e)\mu \left[\frac{S}{\sqrt{3}}a_\e(g_\e)\partial_x^3 (f_\e^n+g_\e^n)+\frac{\sqrt{3 }}{2}\partial_x \sigma_\e(\Gamma_\e^n)\right]^2 +\frac{ a_\e(g_\e^n) \mu}{4}|\partial_x\sigma_\e(\Gamma_\e^n)|^2+\mu D \Phi^{\prime\prime}(\Gamma_\e^n)|\partial_x \Gamma_\e^n|^2 \Bigg\}\,dx.
\end{align*} 
Integrating the above equation with respect to time, yields 
\begin{equation*}  
\mathcal{E}(f_\e^n,g_\e^n,\Gamma_\e^n)(T) + \mathcal{D}_\e(f_\e^n,g_\e^n,\Gamma_\e^n)(T)=  \mathcal{E}(f_0^n,g_0^n,\Gamma_0^n)
\end{equation*}
for all $T\in [0,T_\e^n).$ We deduce that for every $t\in [0,T_\e^n)$, the term $\|\partial_x f_\e^n(t)\|_2^2 = \sum_{k=0}^n  |F_\e^k(t)|^2 \|\partial_x \phi_k\|_2^2$ is bounded by a constant depending  on the initial data, so that $F_\e^k(t)$, and likewise $G_\e^k(t)$, are uniformly bounded for all $t\in[0,T_\e^n)$ and $k\in \{0,\ldots, n\}$.
Furthermore, the energy equality provides the bound of $(\Phi(\Gamma_\e^n))_{n\in \N}$ in $L_\infty(0,T_\e^n;L_1(0,L))$. Using Assumption A2) and $\Phi(1)=\Phi^\prime(1)=0$, we obtain that 
\begin{align*}
\label{phis2} \Phi(s)= \displaystyle{\int_1^s}\displaystyle{\int_1^t}\Phi^{\prime\prime}(u)\,du\,dt \geq \frac{c_\Phi}{2} (s-1)^2\qquad \mbox{for all}\quad s\in \R.
\end{align*}
Hence
\begin{align*}
\displaystyle{\int_0^L}|\Gamma_\e^n(t)|^2\,dx \leq 2\displaystyle{\int_0^L}|\Gamma_\e^n(t)-1|^2+1\,dx \leq \frac{4}{c_\Phi} \displaystyle{\int_0^L}\Phi(\Gamma_\e^n(t))\,dx +2L\leq M^2,
\end{align*}
where $M$ is a constant independent of $t\in [0,T_\e^n)$, $n\in \N$ and $\e\in(0,1]$, which implies that
\begin{equation}\label{ste}
\|\Gamma_\e^n\|_{L_\infty(0,T_\e^n;L_2(0,L))}\leq M,\qquad\mbox{for}\quad n\in\N,\e \in (0,1].
\end{equation}
We show that $v_\e^n$ is uniformly bounded on $[0,T_\e^n)$, which implies the uniform boundedness of  $\Gamma_\e^n$, by $|v_\e^n|\geq |c_\Phi(\Gamma_\e^n-1)|$ (cf. Assumption A2) and $v_\e^n=\Phi^\prime(\Gamma_\e^n)$). 
Invoking Assumption A3), we find that
\[|\Phi^\prime(s)|\leq C_\Phi\left(|s-1|+\left|\frac{|s|^{r+1}}{r+1}-\frac{1}{r+1}\right|\right).\]
 Hence $v_\e^n = \Phi^{\prime}(\Gamma_\e^n)$ is bounded in $L_\infty(0,T_\e^n;L_p(0,L))$ for $p=\frac{2}{r+1}$. 
That is,
\begin{equation*}
\|v_\e^n(t)\|_p^p=\displaystyle{\int_0^L}\left| \displaystyle{\sum_{k=0}^n}V_\e^k(t)\phi_k(x)\right| ^p  \,dx \leq C 
\end{equation*}
for some constant $C>0$, which is independent of $t\in [0,T_\e^n)$, $n\in \N$ and $\e\in(0,1]$.
In view of $\{\phi_k\}_{k\in \N}$ being a Schauder basis in $L_p(0,L)$ and the linear subspace spanned by $\{\phi_0,\cdots,\phi_n\}$ being finite dimensional, the equivalence of norms in finite dimensional spaces implies that
\[\displaystyle{\sum_{k=0}^n}|V_\e^k(t)| ^p\leq c \|v_\e^n(t)\|_p^p \]
for some constant $c>0$. We conclude that $v_\e^n$, and thus $\Gamma_\e^n$, is uniformly bounded on $t\in [0,T_\e^n)$ and the Galerkin approximation $(f_\e^n,g_\e^n,\Gamma_\e^n)$ exists globally.

Furthermore, the mass of each fluid and the surfactant concentration 
is preserved by the Galerkin approximation, which is a consequence of testing the equations in \eqref{rsystem} against the constant function $\phi=1$, integrating by parts and using that  $\partial_x^3 f_\e^n = \partial_x^3 g_\e^n = \partial_x\Gamma_\e^n =0$ at $x=0,L$.
Thus, \eqref{ccc} is satisfied, which completes the proof.
\end{proof}

\subsection{Convergence of the Galerkin Approximations.}
Let $T>0$ be fixed. We show that there exists a weakly converging subsequence of $(f_\e^n,g_\e^n,\Gamma_\e^n)_{n\in \N}$, such that the accumulation point is a weak solution of the regularized problem in the sense of Theorem \ref{T31}. The proof is essentially based on a-priori estimates provided by the energy equality \eqref{energy1} and follows \cite{EM, GW, Wie}.
To proceed, we collect bounds  satisfied by the Galerkin approximation $(f_\e^n,g_\e^n,\Gamma_\e^n)_{n\in \N}$, which are  a consequence of \eqref{energy1} and uniform in $n\in \N$, $\e \in (0,1]$:
\allowdisplaybreaks
\begin{alignat}{2} 
\label{e1} 
& \left\{\partial_x f_\e^n,\partial_x g_\e^n \mid{n\in \N},\e \in (0,1]\right\}  && \qquad\mbox{in} \; L_\infty(0,T; L_2(0,L)), \\[10pt]
\label{e6} 
& \left\{\Phi(\Gamma_\e^n)\mid{n\in \N},\e \in (0,1]\right\} &&\qquad \mbox{in} \; L_\infty(0,T; L_1(0,L)), 
\end{alignat}
\begin{alignat}{2}
\label{e2} 
& \begin{array}{lll} \left\{\sqrt{a_\e(f_\e^n)} \left[\displaystyle{\frac{a_\e(f_\e^n)\partial_x^3(Rf_\e^n+S \mu g_\e^n)}{\sqrt{3}}}\right. \right.\\[10pt]
  \left.\left.\quad\qquad\qquad+\displaystyle{\frac{\sqrt{3}}{2}}\mu\left(S a_\e(g_\e^n)\partial_x^3(f_\e^n+g_\e^n)+\partial_x\sigma_\e(\Gamma_\e^n)\right)\right] \Bigg|{n\in \N},\e \in (0,1]\right\}
  \end{array}
 \; && \mbox{in} \; L_2(\Omega_T),\\[10pt]
\label{e3}
& \left\{ \sqrt{a_\e(f_\e^n)} \left[S a_\e(g_\e^n)\partial_x^3(f_\e^n+g_\e^n)+\partial_x\sigma_\e(\Gamma_\e^n)\right] \Bigg|{n\in \N},\e \in (0,1]\right\} && \mbox{in} \; L_2(\Omega_T),\\[10pt]
\label{e4}
& \left\{\sqrt{a_\e(g_\e^n)}\left[\displaystyle{\frac{S}{\sqrt{3}}}a_\e(g_\e^n)\partial_x^3(f_\e^n+g_\e^n)+\displaystyle{\frac{\sqrt{3}}{2}}\partial_x\sigma_\e(\Gamma_\e^n)\right]\Bigg|{n\in \N},\e \in (0,1]\right\}  \quad &&\mbox{in} \; L_2(\Omega_T),\\[10pt]
\label{e5}
&\left\{\sqrt{a_\e(g_\e^n)}\partial_x \sigma_\e(\Gamma_\e^n)\mid{n\in \N},\e \in (0,1]\right\} && \mbox{in} \; L_2(\Omega_T),\\[10pt]
\label{e7}
&\left\{ \sqrt{\Phi^{\prime\prime}(\Gamma_\e^n)} \partial_x\Gamma_\e^n \mid{n\in \N},\e \in (0,1]\right\}&& \mbox{in} \; L_2(\Omega_T).
\end{alignat}
\vspace{10pt}
Note, that \eqref{e2}--\eqref{e5} also imply  the boundedness of
\begin{alignat}{2}
\label{e21}
& \begin{array}{lll}\left\{\sqrt{a_\e(f_\e^n)} \left[\displaystyle{\frac{a_\e(f_\e^n)\partial_x^3(Rf_\e^n+S \mu g_\e^n)}{\sqrt{3}}}\right. \right.\\[10pt]
  \left.\left.\quad\qquad\qquad+\displaystyle{\frac{2}{\sqrt{3}}}\mu\left(S a_\e(g_\e^n)\partial_x^3(f_\e^n+g_\e^n)+\partial_x\sigma_\e(\Gamma_\e^n)\right)\right] \Bigg|{n\in \N},\e \in (0,1]\right\}
  \end{array}
 \; && \mbox{in} \; L_2(\Omega_T),\\[10pt]
 \label{e23}
& \left\{ a_\e(f_\e^n)^{\frac{3}{2}}\partial_x^3(Rf_\e^n+S \mu g_\e^n)\mid{n\in \N},\e \in (0,1]\right\}\quad && \mbox{in} \; L_2(\Omega_T),\\[10pt]
\label{e41}
& \left\{a_\e(g_\e^n)^{\frac{3}{2}}\partial_x^3(f_\e^n+g_\e^n)\mid{n\in \N},\e \in (0,1]\right\} \quad && \mbox{in} \; L_2(\Omega_T),
\end{alignat}
where \eqref{e21}, \eqref{e23} are a consequence of \eqref{e2}, \eqref{e3} and \eqref{e41} follows from \eqref{e4}, \eqref{e5}.
\begin{lem}\label{GaB} The Galerkin approximation satisfies
\begin{itemize} 
\item[i) ]  $\{f_\e^n, g_\e^n\mid {n\in \N}, \e\in (0,1] \}$ bounded in $ L_\infty(0,T;H^1(0,L))$, where
$\{f_\e^n, g_\e^n\mid {n\in \N} \}$ is additionally bounded in $ L_2(0,T;H^3(0,L))$,
\item[ii) ] $\{\Gamma_\e^n \mid n\in \N,\e\in (0,1]\}$ bounded in $  L_\infty(0,T;L_2(0,L))\cap L_2(0,T;H^1(0,L))$. 
\end{itemize}
\end{lem}
We emphasize that Lemma \ref{GaB} ii) implies the boundedness of
\begin{equation*}
\{\Gamma_\e^n \mid n\in \N,\e\in (0,1]\} \qquad \mbox{in}\quad L_6(\Omega_T),
\end{equation*}
due to \cite[Proposition I.3.2]{DiB}.
\begin{proof}[Proof of Lemma \ref{GaB}]\textbf{}
$i)$ We know from \eqref{e1} that $\{\partial_x f_\e^n,\partial_x g_\e^n\mid n\in \N,\e\in (0,1]\}$ is bounded in $ L_\infty(0,T; L_2(0,L))$. The Poincar\'{e}--Wirtinger Theorem and conservation of mass \eqref{ccc} imply then that
\begin{equation}\label{11}
\{f_\e^n, g_\e^n\mid {n\in \N}, \e\in (0,1] \} \quad \mbox{is bounded in}\quad  L_\infty(0,T;H^1(0,L)).
\end{equation}
For $\e \in (0,1]$ fixed, it follows from \eqref{e23}, \eqref{e41} and the definition of $a_\e$ that
\begin{equation*}
\e^{\frac{3}{2}}\|\partial_x^3(Rf_\e^n+S \mu g_\e^n)\|_{L_2(\Omega_T)} \leq \|a_\e(f_\e^n)^{\frac{3}{2}}\partial_x^3(Rf_\e^n+S \mu g_\e^n)\|_{L_2(\Omega_T)}<c
\end{equation*}
and
\begin{equation*}
\e^{\frac{3}{2}} \|\partial_x^3(f_\e^n+g_\e^n)\|_{L_2(\Omega_T)}\leq \|a_\e(g_\e^n)^{\frac{3}{2}}\partial_x^3(f_\e^n+g_\e^n)\|_{L_2(\Omega_T)} < c
\end{equation*}
for some constant $c>0$ independent of $n\in \N$.
Since $R>S\mu$, we deduce that there exists a constant $c=c(\e)>0$, such that
\begin{equation}\label{p3}
\|\partial_x^3 f_\e^n\|_{L_2(\Omega_T)}, \|\partial_x^3 g_\e^n\|_{L_2(\Omega_T)} <c(\e).
\end{equation}
The Poincar\'{e}--Wirtinger Theorem 
 together with \eqref{11} and \eqref{p3}  yield that
\[ \{f_\e^n, g_\e^n\mid {n\in \N} \} \quad \mbox{ is bounded in}\quad L_2(0,T;H^3(0,L)).\]

$ii)$ In view of \eqref{ste}, 
it is left to show that $\{\Gamma_\e^n \mid n\in \N,\e\in (0,1]\}$ is bounded in  $L_2(0,T;H^1(0,L))$. Note that
\begin{align*}
\|\Gamma_\e^n\|^2_{L_2(0,T;H^1(0,L))} &\leq  T\|\Gamma_\e^n\|^2_{L_\infty(0,T;L_2(0,L))} + \frac{1}{c_\Phi} \displaystyle{\int_0^T} c_\Phi\|\partial_x \Gamma_\e^n(t)\|_2^2\,dt.
\end{align*}
We use Assumption A2) in order to estimate the second term on the right-hand side 
\[\frac{1}{c_\Phi} \displaystyle{\int_0^T} c_\Phi\|\partial_x \Gamma_\e^n(t)\|_2^2\,dt \leq \frac{1}{c_\Phi} \displaystyle{\int_0^T} \|\sqrt{\Phi^{\prime\prime}(\Gamma_\e^n(t))}\partial_x \Gamma_\e^n(t)\|_2^2\,dt,\]
which is bounded by a constant independent of $n\in \N$ and $\e\in(0,1]$, due to \eqref{e7} and the assertion follows. 
\end{proof} 

Notice that the bounds $f_\e^n, g_\e^n \in L_2(0,T,H^3(0,L))$ depend on $\e \in (0,1]$ (cf. \eqref{p3}) and we lose these bounds in the limit when $\e$ tends to zero. 
We make use of the a-priori bounds provided by the energy equality in order to derive uniform bounds for the time derivatives of the Galerkin approximation. Set
\begin{align*}
H_f^{\e,n}:=& \displaystyle{\frac{a_\e(f_\e^n)^2}{\sqrt{3}}}\left[\displaystyle{\frac{a_\e(f_\e^n)\partial_x^3(Rf_\e^n+S \mu g_\e^n)}{\sqrt{3}}}+\frac{\sqrt{3}}{2}\mu\left(S a_\e(g_\e^n)\partial_x^3(f_\e^n+g_\e^n)+ \partial_x\sigma_\e(\Gamma_\e^n)\right)\right] , 
\\[5pt] 
H_g^{\e,n}:=&\displaystyle{\frac{\sqrt{3}}{2}} a_\e(g_\e^n)a_\e(f_\e^n)\left[\displaystyle{\frac{a_\e(f_\e^n)\partial_x^3(Rf_\e^n+S \mu g_\e^n)}{\sqrt{3}}}+\frac{2}{\sqrt{3}}\mu \left(S a_\e(g_\e^n)\partial_x^3(f_\e^n+g_\e^n)+\partial_x\sigma_\e(\Gamma_\e^n)\right)\right] \\[5pt]
  &\quad + \displaystyle{\frac{a_\e^2(g_\e^n)}{\sqrt{3}}}\left[\displaystyle{\frac{S}{\sqrt{3}}}a_\e(g_\e^n)\partial_x^3(f_\e^n+g_\e^n)+\displaystyle{\frac{\sqrt{3}}{2}}\partial_x\sigma_\e(\Gamma_\e^n)\right],
\\[5pt]
H_{\Gamma}^{\e,n}:=& \displaystyle{\frac{\sqrt{3}}{2}}\tau_\e(\Gamma_\e^n) a_\e(f_\e^n)\left[\displaystyle{\frac{a_\e(f_\e^n)\partial_x^3(Rf_\e^n+S \mu g_\e^n)}{\sqrt{3}}}+\frac{2}{\sqrt{3}}\mu\left(S a_\e(g_\e^n)\partial_x^3(f_\e^n+g_\e^n)+\partial_x\sigma_\e(\Gamma_\e^n)\right)\right] \\[5pt]
  & \quad+\displaystyle{\frac{\sqrt{3}}{2}}\tau_\e(\Gamma_\e^n) a_\e(g_\e^n)\left[\displaystyle{\frac{S}{\sqrt{3}}}a_\e(g_\e^n)\partial_x^3(f_\e^n+g_\e^n)+\displaystyle{\frac{\sqrt{3}}{2}}\partial_x\sigma_\e(\Gamma_\e^n)\right] \\[5pt]
  &\quad+\frac{1}{4}\tau_\e(\Gamma_\e^n) a_\e(g_\e^n)\partial_x \sigma_\e(\Gamma_\e^n)-D\partial_x\Gamma_\e^n.
\end{align*}

\begin{lem}\label{GAtD} The time derivatives of the Galerkin approximation satisfy the following bounds:  
\begin{align*}
& \{\partial_t f_\e^n, \partial_t g_\e^n\mid {n\in \N}, \e\in (0,1] \} \quad \mbox{ is bounded in}\quad L_2(0,T;(H^1(0,L))^\prime),\\
 &\{\partial_t \Gamma_\e^n \mid n\in\N, \e\in (0,1]\}\quad \mbox{ is bounded in}\quad L_\frac{3}{2}(0,T;(W^1_3(0,L))^\prime).
 \end{align*}
In fact, for each $\e\in(0,1]$ fixed, we have that
\[
\{\partial_t \Gamma_\e^n \mid n\in\N\}\quad \mbox{ is bounded in}\quad L_2(0,T;(H^1(0,L))^\prime).
\]
\end{lem}

\begin{proof}
Observe that $H_f^{\e,n} \in L_2(\Omega_T)$, since
\allowdisplaybreaks
\begin{align*}
\|H_f^{\e,n}\|_{L_2(\Omega_T)}&= \left\|\displaystyle{\frac{a_\e(f_\e^n)^2}{\sqrt{3}}}\left[\displaystyle{\frac{a_\e(f_\e^n)\partial_x^3(Rf_\e^n+S \mu g_\e^n)}{\sqrt{3}}}+\frac{\sqrt{3}}{2}\mu\left(S a_\e(g_\e^n)\partial_x^3(f_\e^n+g_\e^n)+ \partial_x\sigma_\e(\Gamma_\e^n)\right)\right]\right\|_{L_2(\Omega_T)} \\[5pt]
 &\leq \left\|\displaystyle{\frac{a_\e(f_\e^n)^{\frac{3}{2}}}{\sqrt{3}}}\right\|_{L_\infty(\Omega_T)} \left\|\sqrt{a_\e(f_\e^n)}\left[\displaystyle{\frac{a_\e(f_\e^n)\partial_x^3(Rf_\e^n+S \mu g_\e^n)}{\sqrt{3}}}\right.\right. \\[5pt]
&\left.\left.\qquad\qquad\qquad\qquad\qquad\qquad+\frac{\sqrt{3}}{2}\mu\left(S a_\e(g_\e^n)\partial_x^3(f_\e^n+g_\e^n)+ \partial_x\sigma_\e(\Gamma_\e^n)\right)\right]\right\|_{L_2(\Omega_T)} <c,
\end{align*} 
by \eqref{e2} and Lemma \ref{GaB} i), where $c>0$ is a constant independent of $n\in \N$ and $\e\in(0,1]$.
Analogously, one shows that $H_g^{\e,n}$ is uniformly bounded $L_2(\Omega_T)$ for all $n\in \N$ and $\e\in(0,1]$. Using the uniform bound $\{\Gamma_\e^n\mid {n\in \N}, \e\in (0,1] \}\subset L_6(\Omega_T)$ (cf. Lemma \ref{GaB}),  H\"older's inequality implies that 
\begin{equation}\label{regGamma}\{H_\Gamma^{\e,n}\mid n\in \N, \e\in(0,1]\}\quad \mbox{ is bounded in}\quad L_{\frac{3}{2}}(\Omega_T).
\end{equation} 
Let $\e\in(0,1]$ be fixed. Then, the regularity for $H_\Gamma^{\e,n}$ can be improved due to the regularization by the truncation function $\tau_\e$, which is bounded for every fixed $\e\in (0,1]$:
\begin{align*}
\left\|H_\Gamma^{\e,n}\right\|_{L_2(\Omega_T)} &= \left\| \displaystyle{\frac{\sqrt{3}}{2}}\tau_\e(\Gamma_\e^n) a_\e(f_\e^n)\left[\displaystyle{\frac{a_\e(f_\e^n)\partial_x^3(Rf_\e^n+S \mu g_\e^n)}{\sqrt{3}}}+\frac{2}{\sqrt{3}}\mu\left(S a_\e(g_\e^n)\partial_x^3(f_\e^n+g_\e^n)+\partial_x\sigma_\e(\Gamma_\e^n)\right)\right]\right. \\[5pt]
  & \quad+\displaystyle{\frac{\sqrt{3}}{2}}\tau_\e(\Gamma_\e^n) a_\e(g_\e^n)\left[\displaystyle{\frac{S}{\sqrt{3}}}a_\e(g_\e^n)\partial_x^3(f_\e^n+g_\e^n)+\displaystyle{\frac{\sqrt{3}}{2}}\partial_x\sigma_\e(\Gamma_\e^n)\right] \\[5pt]
  &\left.\quad+\frac{1}{4}\tau_\e(\Gamma_\e^n) a_\e(g_\e^n)\partial_x \sigma_\e(\Gamma_\e^n)-D\partial_x\Gamma_\e^n\right\|_{L_2(\Omega_T)} \\[5pt]
  &\leq \left\| \displaystyle{\frac{\sqrt{3}}{2}}\tau_\e(\Gamma_\e^n)\sqrt{a_\e(f_\e^n)} \right\|_{L_\infty(\Omega_T)} \left\| \sqrt{a_\e(f_\e^n)}\left[\displaystyle{\frac{a_\e(f_\e^n)\partial_x^3(Rf_\e^n+S \mu g_\e^n)}{\sqrt{3}}} \right.\right. \\[5pt]
  &\left.\left.\qquad\qquad\qquad\qquad\qquad\qquad\qquad\qquad+\frac{2}{\sqrt{3}}\mu\left(S a_\e(g_\e^n)\partial_x^3(f_\e^n+g_\e^n)+\partial_x\sigma_\e(\Gamma_\e^n)\right)\right]\right\|_{L_2(\Omega_T)}  \\[5pt]
  & \quad+\left\| \displaystyle{\frac{\sqrt{3}}{2}}\tau_\e(\Gamma_\e^n) \sqrt{a_\e(g_\e^n)}\right\|_{L_\infty(\Omega_T)}\left\|\sqrt{a_\e(g_\e^n)}\left[\displaystyle{\frac{S}{\sqrt{3}}}a_\e(g_\e^n)\partial_x^3(f_\e^n+g_\e^n)+\displaystyle{\frac{\sqrt{3}}{2}}\partial_x\sigma_\e(\Gamma_\e^n)\right]\right\|_{L_2(\Omega_T)} \\[5pt]
  &\quad+\left\|\frac{1}{4}\tau_\e(\Gamma_\e^n) \sqrt{a_\e(g_\e^n)}\right\|_{L_\infty(\Omega_T)}\left\|\sqrt{a_\e(g_\e^n)}\partial_x \sigma_\e(\Gamma_\e^n)\right\|_{L_2(\Omega_T)}+\left\|D\partial_x\Gamma_\e^n\right\|_{L_2(\Omega_T)} <c,
\end{align*}
in view of \eqref{e4}, \eqref{e5}, \eqref{e21} and Lemma \ref{GaB} 
where $c=c(\e)>0$ is a constant dependent on $\e\in(0,1]$, but independent of $n\in \N$.

Given $\xi \in H^1(0,L)$, we use the following notation for the expansion of $\xi$ in the basis $\{\phi_k\mid k\in \N\}$:
\begin{equation}\label{16}
\xi^n:= \displaystyle{\sum_{k=0}^n}\left( \xi \mid \phi_k \right)_2 \phi_k \in \mbox{span}\{\phi_0,\ldots,\phi_n\}.
\end{equation} 
Integration by parts implies
\begin{align*}
\langle \partial_t f_\e^n(t),\xi\rangle_{L_2} &= \left( \partial_t f_\e^n(t)\mid \xi^n \right)_2= \left( H_f^{\e,n}(t)\mid\partial_x \xi^n\right)_2 \leq \|H_f^{\e,n}(t)\|_{L_2(0,L} \|\partial_x \xi^n\|_{L_2(0,L)}\\
&\leq \|H_f^{\e,n}(t)\|_{L_2(0,L} \|\xi\|_{H^1(0,L)}
\end{align*}
for every $t> 0$. 
Hence, the function $\partial_t f_\e^n(t)$ belongs to the dual $(H^1(0,L))^\prime$ of $H^1(0,L)$ for all $t> 0$. Integration with respect to time yields
\begin{align*}\|\partial_t f_\e^n\|^2_{L_2(0,T,(H^1(0,L))^\prime)} &= \displaystyle{\int_0^T}\|\partial_t f_\e^n(t)\|^2_{(H^1(0,L))^\prime}\,dt =\displaystyle{\int_0^T}\sup_{\|\xi\|_{H^1(0,L)}\leq 1} |\left( \partial_t f_\e^n(t)\mid\xi\right)_2 |^2\,dt \\[5pt]
&\leq \displaystyle{\int_0^T}\|H_f^{\e,n}(t)\|_2^2\,dt = \|H_f^{\e,n}\|^2_{L_2(\Omega_T)}.
\end{align*} 
Analogously one shows that $\|\partial_t g_\e^n\|^2_{L_2(0,T,(H^1(0,L))^\prime)} \leq \|H_g^{\e,n}\|^2_{L_2(\Omega_T)}$ and $\|\partial_t \Gamma_\e^n\|^2_{L_2(0,T,(H^1(0,L))^\prime)} \leq \|H_\Gamma^{\e,n}\|^2_{L_2(\Omega_T)}$  for fixed $\e\in(0,1]$ so that
\[\{\partial_t f_\e^n,\partial_t g_\e^n \mid n\in \N,\e \in (0,1]\} \quad \mbox{is bounded in}\quad L_2(0,T,(H^1(0,L))^\prime)\]
and for each $\e\in(0,1]$
\[\{\partial_t \Gamma_\e^n \mid n\in \N\} \quad \mbox{is bounded in}\quad L_2(0,T,(H^1(0,L))^\prime).\]
The remaining assertion that
\[\{\partial_t \Gamma_\e^n \mid n\in \N, \e\in(0,1]\} \quad \mbox{is bounded in}\quad L_\frac{3}{2}(0,T,(W^1_3(0,L))^\prime).\]
follows similarly by recalling that the family $\{H_\Gamma^{\e,n}\mid n\in \N, \e \in (0,1]\}$ is bounded in $L_{\frac{3}{2}}(\Omega_T)$. 

\end{proof}
Let $\e\in (0,1]$ be fixed. 
Lemma \ref{GaB} and Lemma \ref{GAtD} provide necessary bounds for the Galerkin approximation $(f_\e^n, g_\e^n,\Gamma_\e^n)$ to extract weakly convergent subsequences.
Since 
\[H^k(0,L)\chookrightarrow C^{(k-1)+\alpha}([0,L])\hookrightarrow (H^1(0,L))^\prime,\quad L_2(0,T) \chookrightarrow (H^1(0,L))^\prime \] 
for $k\in \N$ and $\alpha \in [0,\frac{1}{2})$,
the bounds of
\begin{alignat*}{2}
&\{f_\e^n, g_\e^n\mid n\in\N\} &&\quad\mbox{in}\quad L_\infty(0,T;H^1(0,L))\cap L_2(0,T;H^3(0,L)),\\[5pt]
&\{\Gamma_\e^n\mid n\in\N\} &&\quad\mbox{in}\quad L_\infty(0,T;L_2(0,L))\cap L_2(0,T;H^1(0,L)),\\[5pt]
&\{\partial_t f_\e^n,\partial_t g_\e^n,\partial_t \Gamma_\e^n\mid n\in\N\}  &&\quad\mbox{in}\quad L_2(0,T;(H^1(0,L))^\prime),
\end{alignat*} 
together with \cite[Corollary 4]{Sim}
imply that
\begin{align}\label{relc}
&(f_\e^n)_{n\in\N}, (g_\e^n)_{n\in\N} \quad \mbox{are relatively compact in}\quad C([0,T];C^{\alpha}([0,L]))\cap L_2(0,T;C^{2+\alpha}([0,L]))\\[5pt]
\label{relcG}
&(\Gamma_\e^n)_{n\in\N} \quad\mbox{is relatively compact in}\quad C([0,T];(H^1(0,L))^\prime)\cap L_2(0,T;C^{\alpha}([0,L]))
\end{align}
for $\alpha \in [0,\frac{1}{2})$.
The relative compactnesses in \eqref{relc} and \eqref{relcG} provide the existence of converging subsequences (not relabeled)
\begin{align}
\label{csf}f_\e^n \longrightarrow f_\e \qquad &\mbox{in} \qquad C([0,T];C^{\alpha}([0,L]))\cap L_2(0,T;C^{2+\alpha}([0,L])), \\[5pt]
\label{csg}g_\e^n \longrightarrow g_\e \qquad &\mbox{in} \qquad C([0,T];C^{\alpha}([0,L]))\cap L_2(0,T;C^{2+\alpha}([0,L])),\\[5pt]
\label{csG}\Gamma_\e^n \longrightarrow \Gamma_\e\qquad &\mbox{in} \qquad C([0,T];(H^1(0,L))^\prime)\cap  L_2(0,T;C^{\alpha}([0,L])).
\end{align}

\begin{lem}\label{lffg} The limit functions $f_\e, g_\e$ obtained in \eqref{csf}, \eqref{csg} belong to
\[  L_\infty(0,T; H^1(0,L))\cap L_2(0,T;H^3(0,L))\]
and there exists a subsequence (not relabeled), such that
\begin{align}\label{3weakC}
\partial_x^k f_\e^n \warrow \partial_x^k f_\e, \qquad \partial_x^k g_\e^n \warrow \partial_x^k g_\e\qquad \mbox{in}\quad L_2(\Omega_T)\quad \mbox{for}\quad k=1,2,3.
\end{align}
Moreover, the time derivatives $\partial_ t f_\e, \partial_t g_\e$ belong to $L_2(0,T;(H^1(0,L))^\prime)$ with
\begin{align*}
\partial_t f_\e^n \warrow \partial_t f_\e, \qquad \partial_t g_\e^n \warrow \partial_t g_\e\qquad \mbox{in}\quad L_2(0,T;(H^1(0,L))^\prime).
\end{align*}
\end{lem}

\begin{proof} We will prove the statements only for $f_\e^n$, the proofs for $g_\e^n$ are similar. 
Owing to Lemma \ref{GaB} i), the sequence $(f_\e^n)_{n\in \N}$ is bounded in $L_2(0,T;H^3(0,L))$. Thus, by Eberlein--Smulyan's theorem, there exists a weakly convergent subsequence (not relabeled), such that
\begin{equation}\label{weakC}
f_\e^n\warrow \overline{f}_\e\qquad \mbox{in}\quad L_2(0,T;H^3(0,L))
\end{equation}
for some $\overline{f}_\e\in L_2(0,T;H^3(0,L))$. The uniqueness of limits in the sense of distributions implies together with \eqref{csf} that $f_\e = \overline{f}_\e\in L_2(0,T;H^3(0,L))$ and the claim \eqref{3weakC} is satisfied.
By the weak-* compactness of $L_\infty(0,T;H^1(0,L))$ and \eqref{csf}, we deduce that the limit function $f_\e$ belongs to $L_\infty(0,T;H^1(0,L))$.
In view of Lemma \ref{GAtD}, the time derivative $(\partial_t f_\e^n)_{n\in \N}$ is bounded in the Hilbert space $L_2(0,T;(H^1(0,L))^\prime)$. Thus, by Eberlein--Smulyan's theorem, there exists a weakly convergent subsequence (not relabeled) 
\[\partial_t f_\e^{n} \warrow h\qquad \mbox{in}\quad L_2(0,T;(H^1(0,L))^\prime),\]
for some limit function $h\in L_2(0,T;(H^1(0,L))^\prime)$. The identification of the limit function $h$ with $\partial_t f_\e$ is then a consequence of \eqref{csf}
\end{proof}

\begin{bem}\label{fepsu} 
\emph{Note that the bounds of $(f_\e^n)_{n\in \N}, (g_\e^n)_{n\in \N}$ in $L_\infty(0,T;H^1(0,L))$  (cf. Lemma \ref{GaB}) are in fact uniform also in $\e\in (0,1]$ and we conclude that
}
\begin{equation}\label{remark}
\{f_\e, g_\e\mid \e\in (0,1]\}\quad\mbox{\emph{is bounded in}}\quad L_\infty(0,T;H^1(0,L)).
\end{equation}
\emph{This uniform bound will be in particular  necessary in the proof of Theorem \ref{nonnG}, where the non-negativity of the family of Galerkin approximations $(\Gamma_\e)_{\e \in (0,1]}$ is shown.}
\end{bem}

In the following lemma we collect weak and strong convergences concerning the family $(\Gamma_\e^n)_{n\in \N}$ in certain Banach spaces.
Note that due to $\Omega_T$ having finite measure it is a consequence of H\"older's inequality that any bound in $L_p(\Omega_T)$  holds true for the whole range $[1,p]$.
\begin{lem}\label{lfG}
The family $(\Gamma_\e^n)_{n\in \N}$ satisfies
\begin{itemize}
\item[i)] $\Gamma_\e^n \longrightarrow \Gamma_\e$ in $ L_q(\Omega_T)$ for all $q\in [1,6)$ and the family $\{\Gamma_\e\mid \e\in (0,1]\}$ is bounded in 
$L_\infty(0,T;L_2(0,L))\cap L_2(0,T;H^1(0,L))\subset L_6(\Omega_T)$.
\item[ii)] There exist a subsequence (not relabeled) such that $\partial_x \Gamma_\e^n \warrow \partial_x \Gamma_\e$ in $ L_2(\Omega_T)$.
\item[iii)] For each $\e\in(0,1]$ fixed we have that $\partial_t \Gamma_\e^n \warrow \partial_t \Gamma_\e $ in $ L_{2}(0,T;(H^1(0,L))^\prime).$ Moreover, the family $\{\partial_t \Gamma_\e\mid \e\in(0,1]\}$ is bounded in $ L_{\frac{3}{2}}(0,T;(W^1_3(0,L))^\prime)$.
\item[iv)]  $\Phi(\Gamma_\e^n)\longrightarrow\Phi(\Gamma_\e) $ in $ L_1(\Omega_T)$ and the family $\{\Phi(\Gamma_\e)\mid \e \in(0,1]\}$ is bounded in $L_{\frac{3}{2}}(\Omega_T)$. 
\item[v)] There exist a subsequence (not relabeled) such that $ \sqrt{\Phi^{\prime\prime}(\Gamma_\e^n)}\partial_x \Gamma_\e^n \warrow \sqrt{\Phi^{\prime\prime}(\Gamma_\e)}\partial_x \Gamma_\e $ in $  L_2(\Omega_T)$. 
\end{itemize}
\end{lem}

\begin{proof}\textbf{}
i) Since $L_6(\Omega_T)$ is a reflexive Banach space, we can extract, by Eberlein--Smulyan's theorem, a weakly convergent subsequence (not relabeled) with
\begin{equation*}
\Gamma_\e^n \warrow \Gamma_\e \qquad \mbox{in}\quad L_6(\Omega_T),
\end{equation*}
where the identification of the limit function is a consequence of \eqref{csG}.
Using Riesz's interpolation theorem, we obtain
\begin{align*}
\|\Gamma_\e^n - \Gamma_\e\|_{L_q(\Omega_T)} \leq \|\Gamma_\e^n -\Gamma_\e\|_{L_p(\Omega_T)}^{1-\theta}\|\Gamma_\e^n -\Gamma_\e\|_{L_l(\Omega_T)}^{\theta}
\end{align*}
for $ \theta\in [0,1]$ and $\frac{1}{q}= \frac{1-\theta}{p}+\frac{\theta}{l}$. Choosing $l=2$ and $p=6$ it follows from Lemma \ref{GaB} and \eqref{csG} that
\begin{equation*}
\|\Gamma_\e^n - \Gamma_\e\|_{L_q(\Omega_T)} \leq (\|\Gamma_\e^n\|_{L_6(\Omega_T)}+\|\Gamma_\e\|_{L_6(\Omega_T)})^{1-\theta} \|\Gamma_\e^n-\Gamma_\e\|_{L_2(\Omega_T)}^\theta \longrightarrow 0
\end{equation*}
for $q= \frac{6}{1+2\theta}\in [2,6)$ and $n\longrightarrow \infty$.  Due to Lemma \ref{GaB} ii) and Eberlein--Smulyan's theorem, there exist a convergent subsequence (not relabeled) such that
\[\Gamma_\e^n \warrow \Gamma_\e\quad \mbox{in}\quad L_\infty(0,T;L_2(0,L))\cap L_2(0,T;H^1(0,L)) \subset L_6(\Omega_T),\]
where the identification of the limit is due to \eqref{csG}. Since the norm is weak lower semi-continuous, we deduce that the family $\{\Gamma_\e\mid \e\in (0,1]\}$ is bounded in 
$L_\infty(0,T;L_2(0,L))\cap L_2(0,T;H^1(0,L))$.

ii) and iii) are a consequence of Lemma \ref{GaB}, Lemma \ref{GAtD} and Eberlein--Smulyan's theorem, where the identification of of the limits is due to \eqref{csG}.

iv)
We deduce from i) that there exists a subsequence (not relabeled) such that $\Gamma_\e^n - \Gamma_\e \longrightarrow 0$ point-wise almost everywhere. Since $\Phi$ is continuous, also the function $\Phi(\Gamma_\e^n)-\Phi(\Gamma_\e)$ converges point-wise to zero almost everywhere. In view of Lemma \ref{GaB} and i), Assumption A3) implies that
\[\{\Phi(\Gamma_\e^n)-\Phi(\Gamma_\e)\mid n\in\N, \e\in(0,1]\} \qquad \mbox{is bounded in} \quad L_{\frac{3}{2}}(\Omega_T).\]
Noting that any function belonging to $L_p(\Omega_T)$, where $1<p<\infty$ is uniformly integrable and $\Omega_T$ has finite measure, Vitali's convergence theorem guarantees that
\[\Phi(\Gamma_\e^n)\longrightarrow \Phi(\Gamma_\e)\qquad \mbox{in}\quad L_1(\Omega_T).\]

v) Using Assumption A3) the same argument as in iv) proves that $\{\sqrt{\Phi^{\prime\prime}(\Gamma_\e^n)}-\sqrt{\Phi^{\prime\prime}(\Gamma_\e)}\mid \e\in[0,1] \}$ is bounded in $L_6(\Omega_T)$ and 
\begin{equation}
\label{Phi_prime}
\sqrt{\Phi^{\prime\prime}(\Gamma_\e^n)}\longrightarrow\sqrt{\Phi^{\prime\prime}(\Gamma_\e)} \quad \mbox{in}\quad L_q(\Omega_T)\qquad \mbox{for}\quad q\in[1,6).
\end{equation}
From \eqref{e7} we deduce that $(\sqrt{\Phi^{\prime\prime}(\Gamma_\e^n)}\partial_x \Gamma_\e^n)_{n\in \N}$ is bounded in $L_2(\Omega_T)$. 
Hence, there exists a weakly convergent subsequence (not relabeled) such that 
\[\sqrt{\Phi^{\prime\prime}(\Gamma_\e^n)}\partial_x \Gamma_\e^n \warrow \sqrt{\Phi^{\prime\prime}(\Gamma_\e)}\partial_x \Gamma_\e \qquad \mbox{in}\quad  L_2(\Omega_T),\]
where the identification of the limit is due to Lemma \ref{App} and the uniform boundedness of the limit family $(\Gamma_\e)_{\e\in(0,1]}$ due to the lower semi-continuity of the norm.
\end{proof}

\begin{lem}\label{surfc}
The Galerkin approximation  $(\Gamma_\e^n)_{n\in \N}$ contains a subsequence (not relabeled), such that
\[ \partial_x \sigma_\e(\Gamma_\e^n)\warrow \partial_x \sigma_\e(\Gamma_\e)\quad\mbox{in}\quad  L_s(\Omega_T)\qquad \mbox{for}\quad s\in \Big[1,\frac{6}{5}\Big).\]
\end{lem}

\begin{proof} Recall that in virtue of \eqref{p} and \eqref{taudef} we can write 
\[\partial_x \sigma_\e(\Gamma_\e^n)= -\tau_\e(\Gamma_\e^n) \sqrt{\Phi^{\prime\prime}(\Gamma_\e^n)}\sqrt{\Phi^{\prime\prime}(\Gamma_\e^n)}\partial_x\Gamma_\e^n.\]  
 By construction we have that $|\tau_\e(s)|\leq |s|$ for all $s\in \R$ (cf. \eqref{abtau}). As a consequence of Lemma \ref{GaB} and Lemma \ref{lfG} i), we find that $(\tau_\e(\Gamma_\e^n)-\tau_\e(\Gamma_\e))_{n\in \N}$ is uniformly bounded in $L_6(\Omega_T)$ and there exists a subsequence (not relabeled) such that   $\tau_\e(\Gamma_\e^n)-\tau_\e(\Gamma_\e)$ converges to zero almost everywhere. The same argument as in the proof of Lemma \ref{lfG} iv) and v) yields that
\begin{equation}\label{tauG}
\tau_\e(\Gamma_\e^n)\longrightarrow \tau_\e(\Gamma_\e)\quad\mbox{in}\quad L_p(\Omega_T)\quad\mbox{for}\quad p\in [1,6).
\end{equation}
Owing to  \eqref{Phi_prime} and \eqref{tauG}, H\"older's inequality implies that
\begin{equation}\label{25}
\tau_\e(\Gamma_\e^n)\sqrt{\Phi^{\prime\prime}(\Gamma_\e^n)} \longrightarrow \tau_\e(\Gamma_\e) \sqrt{\Phi^{\prime\prime}(\Gamma_\e)}\qquad\mbox{in}\quad  L_m(\Omega_T)\quad \mbox{for}\quad m\in  [1,3).
\end{equation}
Recalling that the energy equality provides the bound of $(\sqrt{\Phi^{\prime\prime}(\Gamma_\e^n)}\partial_x \Gamma_\e^n)_{n\in\N}$ in $L_2(\Omega_T)$ (cf. Lemma \ref{lfG}), we apply again Hölder's inequality to obtain that 
\begin{equation*}\label{23}
\left( \partial_x\sigma(\Gamma_\e^n) \right)_{n\in \N} \quad \mbox{is bounded in}\quad L_{\frac{6}{5}}(\Omega_T).
\end{equation*}
 Lemma \ref{App}, Lemma \ref{lfG} v) and \eqref{25} imply that there exists a subsequence (not relabeled), such that
\begin{equation*}
\partial_x \sigma(\Gamma_\e^n)\warrow \partial_x \sigma(\Gamma_\e) \quad\mbox{in}\quad L_s(\Omega_T)\qquad \mbox{for}\quad s\in \Big[1,\frac{6}{5}\Big).
\end{equation*}
\end{proof}
Since $(f_\e^n, g_\e^n)(t)$ tends towards $(f_\e, g_\e)(t)$ in $\left(C^{\alpha}([0,L])\right)^2$, by \eqref{csf}, \eqref{csg}, for every $t\geq 0$ and $\alpha\in [0,\frac{1}{2})$, the initial conditions
\begin{equation*}\label{21}
f_\e(0) = f^0,\qquad g_\e(0) = g^0
\end{equation*}
are satisfied and 
\begin{equation*}
\|f_\e(t)\|_1 = \|f^0\|_1,\qquad \|g_\e(t)\|_1 = \|g^0\|_1
\end{equation*}
for all $t\geq 0$. In view of  \eqref{csG}, we obtain $\Gamma_\e^n(0)\longrightarrow \Gamma_\e(0)$ in $(H^1(0,L))^\prime$. Recall that by definition and construction of the Galerkin approximation 
\[\Gamma_\e^n(0)= W(v_\e^n(0))= W(v_0^n)\longrightarrow W(v^0)=\Gamma_0\] 
in $L_p(\Omega_T)$ for $p= 2(r+1)$. Hence, we deduce that the initial condition $\Gamma_\e(0)= \Gamma^0$ is satisfied.
By \eqref{csG}, $(\Gamma_\e^n)_{n\in\N}$ converges towards $\Gamma_\e$  in $L_2(0,T;C^\alpha([0,L]))$, which implies the existence of a further subsequence of $(\Gamma_\e^n)_{n\in\N}$ (not relabeled) such that  $\Gamma_\e^n(t)\longrightarrow \Gamma_\e(t)$ for almost every $t\geq 0$ in $C^\alpha([0,L])$. Therefore, 
\begin{equation*}
\|\Gamma_\e(t)\|_1 = \|\Gamma^0\|_1\quad \mbox{for almost all} \quad t\geq 0.
\end{equation*}
We prove that the energy inequality is satisfied for the limit $(f_\e,g_\e,\Gamma_\e)$ of the Galerkin approximation. Since by construction $a_\e$ and $\tau_\e$ are locally Lipschitz continuous,
\eqref{csf}--\eqref{csG} imply that 
\begin{align}
\label{acsf}a_\e(f_\e^n) \longrightarrow a_\e(f_\e) \qquad &\mbox{in} \quad C(\Omega_T), \\[5pt]
\label{acsg}a_\e(g_\e^n) \longrightarrow a_\e(g_\e) \qquad &\mbox{in} \quad C(\Omega_T),\\[5pt]
\label{acsG}\tau_\e(\Gamma_\e^n) \longrightarrow \tau_\e(\Gamma_\e)\qquad &\mbox{in} \quad L_2(0,T;C^{\alpha}([0,L]))
\end{align}
for $\alpha \in [0,\frac{1}{2})$.
The energy equality provides that $(\sqrt{a_\e(g_\e^n)}\partial_x \sigma_\e(\Gamma_\e^n))_{n\in\N}$ is bounded in $L_2(\Omega_T)$   (cf.\eqref{e5}). Lemma \ref{surfc}, \eqref{acsg} and Lemma \ref{App}  imply then the existence of a weakly convergent subsequence (not relabeled), so that 
\begin{align}\label{ee5}
\sqrt{a_\e(g_\e^n)}\partial_x \sigma_\e(\Gamma_\e^n) \warrow \sqrt{a_\e(g_\e)}\partial_x \sigma_\e(\Gamma_\e)\qquad \mbox{in}\quad L_2(\Omega_T).
\end{align}
As a consequence of \eqref{e5}, \eqref{acsf}, \eqref{acsg}, and  $a_\e(s)\geq \e>0$ for all $s\in \R$, we obtain that
\[\left(\sqrt{a_\e(f_\e^n)}\partial_x \sigma_\e(\Gamma_\e^n) )\right)_{n\in\N}=\left(\frac{\sqrt{a_\e(f_\e^n)}}{\sqrt{a_\e(g_\e^n)}}\sqrt{a_\e(g_\e^n)}\partial_x \sigma_\e(\Gamma_\e^n)\right)_{n\in\N}\quad\mbox{is bounded in }\quad L_2(\Omega_T).\] Hence, as before, Lemma \ref{surfc}, \eqref{acsf} and Lemma \ref{App} imply that there exists a subsequence (not relabeled), so that
\begin{align}\label{ee52}
\sqrt{a_\e(f_\e^n)}\partial_x \sigma_\e(\Gamma_\e^n) \warrow \sqrt{a_\e(f_\e)}\partial_x \sigma_\e(\Gamma_\e)\qquad \mbox{in}\quad L_2(\Omega_T).
\end{align}
We deduce that there exist weakly convergent subsequences (not relabeled) with
\begin{align}
\label{ee3}
& \sqrt{a_\e(f_\e^n)} \left[S a_\e(g_\e^n)\partial_x^3(f_\e^n+g_\e^n)+\partial_x\sigma_\e(\Gamma_\e^n)\right]\warrow \sqrt{a_\e(f_\e)} \left[S a_\e(g_\e)\partial_x^3(f_\e+g_\e)+\partial_x\sigma_\e(\Gamma_\e)\right]\\[10pt]
\label{ee2} 
 &\sqrt{a_\e(f_\e^n)} \left[\displaystyle{\frac{a_\e(f_\e^n)\partial_x^3(Rf_\e^n+S \mu g_\e^n)}{\sqrt{3}}}+\displaystyle{\frac{\sqrt{3}}{2}}\mu \left(S a_\e(g_\e^n)\partial_x^3(f_\e^n+g_\e^n)+\partial_x\sigma_\e(\Gamma_\e^n)\right)\right]\warrow J_f^\e,\\[10pt]
\label{ee4}
& \sqrt{a_\e(g_\e^n)}\left[\displaystyle{\frac{S}{\sqrt{3}}}a_\e(g_\e^n)\partial_x^3(f_\e^n+g_\e^n)+\displaystyle{\frac{\sqrt{3}}{2}}\partial_x\sigma_\e(\Gamma_\e^n)\right]\warrow J_g^\e\\[10pt]
\label{eE5}
&\sqrt{a_\e(f_\e^n)} \left[\displaystyle{\frac{a_\e(f_\e^n)\partial_x^3(Rf_\e^n+S \mu g_\e^n)}{\sqrt{3}}}+\displaystyle{\frac{2}{\sqrt{3}}}\mu \left(S a_\e(g_\e^n)\partial_x^3(f_\e^n+g_\e^n)+\partial_x\sigma_\e(\Gamma_\e^n)\right)\right]\warrow J_{f,g}^\e
\end{align}
in $L_2(\Omega_T)$,
by means of Lemma \ref{lffg}, Lemma \ref{App} \eqref{acsf}, \eqref{acsg}, \eqref{ee5} and \eqref{ee52}. 
We conclude that, owing to Lemma \ref{lfG} v) and \eqref{ee5}, \eqref{ee3}--\eqref{eE5} there exists a subsequence (not relabeled), so that
\begin{align}\label{De1}
\mathcal{D}_\e(f_\e^n,g_\e^n,\Gamma_\e^n)(T) \warrow \mathcal{D}_\e(f_\e,g_\e,\Gamma_\e)(T)\qquad\mbox{for all}\quad T>0.
\end{align}
Moreover,  $\mathcal{E}(f_\e^n,g_\e^n,\Gamma_\e^n)(T)\longrightarrow \mathcal{E}(f_\e,g_\e,\Gamma_\e)(T)$ for almost all $T>0$, by \eqref{csf}, \eqref{csg} and Lemma \ref{lfG} iv), so that, view of \eqref{De1} and the norm being lower semi-continuous, the energy inequality \eqref{energy} holds.

To finish the proof of Theorem \ref{T31}, it remains show that \eqref{T1}--\eqref{T3} are satisfied. Let $\xi\in L_2(0,T;H^1(0,L))$ be given. As in \eqref{16}, for each $n\in \N$ the expansion of $\xi$ is given by
\[\xi^n(t,\cdot):= \displaystyle{\sum_{k=0}^n}\left( \xi(t,\cdot)\mid \phi_k\right)_2\phi_k,\qquad t\in (0,T).\]
Integration by parts implies that
\begin{equation}\label{31neu}
\displaystyle{\int_0^T}\langle \partial_t f_\e^n(t), \xi^n(t)\rangle\,dt = \displaystyle{\int_{\Omega_T}}H_f^{\e,n}\partial_x \xi^n \, d(x,t)
\end{equation}
for every $n\in \N$. Next, we show that we can pass to the limit $n\rightarrow \infty$ in \eqref{31neu}, after possibly extracting a further subsequence.
Observe that \eqref{acsf}, \eqref{ee2} and Lemma \ref{App} imply  that 
\begin{equation*}
H_f^{\e,n}\warrow H_f^{\e}\quad \mbox{in}\quad L_2(\Omega_T),
\end{equation*}
where $H_f^{\e}$ is given by
\begin{align*}
H_f^{\e}&:= \displaystyle{\frac{a_\e(f_\e)^\frac{3}{2}}{\sqrt{3}}}J_f^\e.
\end{align*}
Since, by  the Lebesgue dominated convergence theorem, $\xi^n\rightarrow\xi$ in $L_2(0,T;H^1(0,L))$, we obtain that
\begin{equation}\label{32}
\displaystyle{\int_{\Omega_T}} H_f^{\e,n} \partial_x\xi^n\,d(x, t)\longrightarrow \displaystyle{\int_{\Omega_T}} H_f^{\e} \partial_x\xi \,d(x,t)
\end{equation}
and in view of Lemma \ref{lffg}
\begin{equation}\label{32abc}
\displaystyle{\int_0^T}\langle \partial_t f_\e^n(t), \xi^n(t) \rangle \,dt \longrightarrow \displaystyle{\int_0^T} \langle \partial_t f_\e(t), \xi(t) \rangle_{H^1}\,dt .
\end{equation}
Thus, \eqref{T1} is satisfied in virtue of \eqref{31neu}--\eqref{32abc}. Using   \eqref{acsf}, \eqref{acsg}, \eqref{ee3} and \eqref{ee4} we find that
\begin{equation*}\label{33}
H_g^{\e,n}\warrow H_g^{\e}\quad \mbox{in}\quad L_2(\Omega_T),
\end{equation*}
where $H_g^\e$ is given by 
\begin{align*}
H_g^\e &:=\displaystyle{\frac{\sqrt{3}}{2}} a_\e(g_\e)\sqrt{a_\e(f_\e)} J_{f,g}^\e+ \frac{a_\e(g_\e)^{\frac{3}{2}}}{\sqrt{3}}J_g^\e.
\end{align*}
Analogously, we obtain that
\begin{equation*}\label{332}
H_\Gamma^{\e,n}\warrow H_\Gamma^{\e}\qquad \mbox{in}\quad L_2(\Omega_T),
\end{equation*}
where the limit function 
\begin{align*}
H_{\Gamma}^{\e}&:= \displaystyle{\frac{\sqrt{3}}{2}}\tau_\e(\Gamma_\e) \sqrt{a_\e(f_\e)}J_{f,g}^\e+\displaystyle{\frac{\sqrt{3}}{2}}\tau_\e(\Gamma_\e) \sqrt{a_\e(g_\e)}J_g^\e+\frac{1}{4}\tau_\e(\Gamma_\e) a_\e(g_\e^n)\partial_x \sigma_\e(\Gamma_\e)-D\partial_x\Gamma_\e
\end{align*}
can be identified in view of \eqref{acsf}--\eqref{acsG}, \eqref{ee52}, \eqref{ee2}, \eqref{ee4} and Lemma \ref{lfG}. Passing to the limit as in \eqref{31neu}, we deduce that \eqref{T2} and \eqref{T3} are satisfied and the proof of Theorem \ref{T31} is complete.

\section{Existence and Non-Negativity of Weak Solutions for the Original System}\label{EGWS} In this section we prove the main result Theorem \ref{MT}. We use the global weak solutions $(f_\e,g_\e,\Gamma_\e)_{\e \in (0,1]}$ of the regularized problem \eqref{rsystem} to find, in the limit $\e \searrow 0$, global weak solutions of the original problem \eqref{system2}. We emphasize that in the sequel, the initial data $f^0, g^0, \Gamma^0$  are non-negative. 
Following \cite{EM} we show that if $(\e_k)_{k\in \N}\subset (0,1]$ is such that $\e_k\searrow 0$ for $k\longrightarrow \infty$ and there exist functions $f,g\in C(\overline{\Omega}_T)$ with
\begin{equation}\label{ap}
f_{\e_k} \longrightarrow f,\qquad g_{\e_k} \longrightarrow g\qquad \mbox{in}\quad C(\overline{\Omega}_T)\quad \mbox{for}\quad k\rightarrow \infty,
\end{equation}
then the accumulation points $f,g$ are non-negative. Concerning the sequence $(\Gamma_\e)_{\e\in (0,1]}$, we use the idea in \cite{EW3} to prove that already $(\Gamma_\e)_{\e \in (0,1]}\geq 0$, so that if there exists a limit function of $(\Gamma_\e)_{\e\in (0,1]}$ for $\e\searrow 0$, the almost everywhere non-negativity of the accumulation point will be inherited.

\subsection{Non-Negativity of Accumulation Points of the Solutions to the Regularized Systems.}

Let $(\e_k)_{k\in \N}\subset (0,1]$ be such that $\e_k\searrow 0$ for $k\longrightarrow \infty$ and assume there exist functions $f,g\in C(\overline{\Omega}_T)$, such that \eqref{ap} is satisfied.
In order to show that for non-negative initial data $f^0,g^0$ the accumulation points $(f,g)$ as in \eqref{ap} satisfy the non-negativity property, we define in analogy to \cite{EM} a function $\psi \in C^\infty (\R)$, which is non-negative, supported in $[-1,0]$ and satisfies
\begin{equation*}
\displaystyle{\int_\R}\psi(x)\, dx=\displaystyle{\int_{-1}^0}\psi(x)\, dx=1.
\end{equation*}
Further, let $\chi_1:\R\longrightarrow\R$ be defined by
\[\chi_1(x):= \displaystyle{\int_x^0}\displaystyle{\int_s^\infty}\psi(\tau)\,d\tau\,ds\qquad \mbox{for}\quad x\in \R\]
and 
\begin{equation}\label{mollifier}
\chi_\delta(x):= \delta \chi_1\left(\frac{x}{\delta}\right).
\end{equation}
The function $\chi_\delta$ is a smooth approximation of $\max\{-\,\cdot, 0\}$ as $\delta\rightarrow 0$.
We deduce easily from the definition of $\chi_\delta$, that the  following properties hold true:
\begin{lem}\label{chi} The function $\chi_\delta$ satisfies
\begin{itemize}
\item[i) ] $\|\chi_\delta - \mbox{\emph{max}}\,\{-\Id,0\}\|_\infty \leq \delta,$
\item[ii) ]$ \|\chi_\delta^\prime\|_\infty \leq 1$ and  $\|\chi_\delta^{\prime\prime}\|_\infty \leq \delta^{-1}\|\psi\|_\infty,$
\item[iii) ] $|s\chi_\delta^{\prime\prime}(s)|\leq K$ for all $s\in [-\delta,\delta]$, where $K:= \|\psi\|_\infty$,
\item[iv) ] $\chi_\delta^{\prime\prime}(s)=0$ on $\R\setminus [-\delta,0]$.
\end{itemize}
\end{lem}

The following lemma will play the key role in proving the non-negativity of $f$ and $g$.
\begin{lem}\label{chilemma} 
There exists a constant $c>0$, independent of $\e\in (0,1]$ and $t\geq 0$, such that the solutions of the regularized system $f_\e$ and $g_\e$ satisfy
\begin{equation}\label{40}
\left|\displaystyle{\int_0^L} \chi_{\sqrt{\e}}(f_\e(t))\, dx \right|\leq c \sqrt{t}\e,\qquad \left|\displaystyle{\int_0^L} \chi_{\sqrt{\e}}(g_\e(t))\, dx \right|\leq c \sqrt{t}\e
\end{equation}
for all $\e \in (0,1]$ and $t\geq 0$.
\end{lem}

\begin{proof} 
Let $\delta>0$.
The statement is true for $t=0$, since $f^0,g^0$ are assumed to be non-negative. By \cite[Lemma 7.5]{GT}, the composition $\chi^{\prime}_\delta (f_\e)$ belongs to $L_2(0,T;H^1(0,L))$.
 Notice that formally 
\begin{align}\label{39aa}
\begin{split}
&\frac{d}{dt}\displaystyle{\int_0^L}\chi_\delta (f_\e)(t)\,dx = \left\langle \chi^{\prime}_\delta (f_\e)(t),\partial_t f_\e(t)\right\rangle, \\[5pt]
&\frac{d}{dt}\displaystyle{\int_0^L}\chi_\delta (g_\e)(t)\,dx = \left\langle \chi^{\prime}_\delta (g_\e)(t), \partial_t g_\e(t)\right\rangle.
\end{split}
\end{align}
Hence, integrating \eqref{39aa} with respect to time, we get
\begin{align}\label{39}
\begin{split}
&\displaystyle{\int_0^L} \chi_\delta(f_\e(T))\, dx = \displaystyle{\int_{\Omega_T}}H_f^\e\chi_\delta^{\prime\prime}(f_\e)\partial_x f_\e \,d(x,t) ,\\[5pt] 
&\displaystyle{\int_0^L} \chi_\delta(g_\e(T))\, dx = \displaystyle{\int_{\Omega_T}}H_g^\e\chi_\delta^{\prime\prime}(g_\e)\partial_x g_\e \,d(x,t) 
\end{split} 
\end{align} 
for all $T> 0$. The identities \eqref{39} will be justified below. Assume for the present moment that \eqref{39} holds true.
Since $\chi_\delta^{\prime\prime}=0$ on $\R\setminus [-\delta,0]$, the Hölder inequality implies that
\begin{align*}
\left(\displaystyle{\int_0^L} \chi_\delta(f_\e(T))\, dx \right)^2 &\leq \left( \displaystyle{\int_{[-\delta \leq f_\e \leq 0]}}H_f^\e\chi_\delta^{\prime\prime}(f_\e)\partial_x f_\e \,d(x,t)\right)^2 \\[5pt]
&\leq  \displaystyle{\int_{[-\delta \leq f_\e \leq 0]}} \left| J_f^\e  \right|^2 \,d(x,t)  \times \displaystyle{\int_{[-\delta \leq f_\e \leq 0]}}\frac{a_\e(f_\e)^3}{3}|\chi_\delta^{\prime\prime}(f_\e)|^2|\partial_x f_\e|^2 \,d(x,t).
\end{align*}
Choosing $\delta:= \sqrt{\e}$ and recalling that $a_\e=\e$ on $(-\infty,0]$, the energy equality \eqref{energy} together with Lemma \ref{chi} ii) imply the existence of a constants $c, C>0$, independent of $\e\in (0,1]$ and $T>0$, so that
\begin{align*}
\left|\displaystyle{\int_0^L} \chi_{\sqrt{\e}}(f_\e(T))\, dx\right| &\leq c \left(  \displaystyle{\int_{[-\sqrt{\e} \leq f_\e \leq 0]}}\frac{a_\e^3(f_\e)}{3}|\chi_{\sqrt{\e}}^{\prime\prime}(f_\e)|^2|\partial_x f_\e|^2 \,d(x,t) \right)^{\frac{1}{2}} \\[5pt]
& \leq c \e \|\psi\|_\infty \left( \displaystyle{\int_{\Omega_T}}|\partial_x f_\e|^2 \,d(x,t) \right)^{\frac{1}{2}} \leq C \sqrt{T}\e,
\end{align*}
which is the desired estimate for $f_\e$ in \eqref{40}. Using a similar argument we prove the statement for $g_\e$. 
We are left to show that \eqref{39} holds true. Consider for $t>0$
\begin{equation*}
\frac{d}{dt} \displaystyle{\int_0^L} \chi_\delta(f_\e^n(t))\, dx = \displaystyle{\int_0^L} \chi_\delta^\prime(f_\e^n(t))\partial_t f_\e^n(t)\, dx,
\end{equation*} 
where $(f^n)_{n\in \N}$ is the Galerkin approximation of the previous section.
Since $\chi_\delta^\prime(f_\e^n(t))$ belongs to $H^1(0,L)$ for all $t>0$, 
we can use its Fourier expansion as a test function for $\partial_t f_\e^n$ and find that
\begin{align*}
\frac{d}{dt} \displaystyle{\int_0^L} \chi_\delta(f_\e^n(t))\, dx &=\displaystyle{\int_0^L}\partial_t f_\e^n(t) \displaystyle{\sum_{k=0}^n}\left( \chi^\prime_\delta(f_\e^n(t)) \mid \phi_k \right)_2 \phi_k \, dx\\[5pt]
& =\displaystyle{\int_0^L}H_f^{\e,n}(t) \partial_x \left(\displaystyle{\sum_{k=0}^n} \left( \chi^\prime_\delta(f_\e^n(t)) \mid \phi_k \right)_2 \phi_k \right)\, dx.
\end{align*}
Integration with respect to time yields
\begin{equation}\label{42}
\displaystyle{\int_0^L} \chi_\delta(f_\e^n(T))\, dx =  \displaystyle{\int_{\Omega_T}} H_f^{\e,n} \displaystyle{\sum_{k=0}^n}\left( \chi^\prime_\delta(f_\e^n(t)) \mid \phi_k \right)_2 \partial_x\phi_k \, d(x,t)
\end{equation}
for all $T>0$. Since the function $\chi_\delta$ is  continuous and $f_\e^n(t)\longrightarrow f_\e(t)$ point-wise for every $t>0$, the left-hand side of \eqref{42} tends to $\int_0^L \chi_\delta(f_\e(T))\, dx$. 
Investigating the convergence of the right-hand side of \eqref{42}, observe first that
\begin{align}
\begin{split}\label{first}
\chi^{\prime\prime}_\delta(f_\e)(t)\partial_x f_\e - \displaystyle{\sum_{k=0}^n}\left( \chi^\prime_\delta(f_\e^n(t)) \mid \phi_k \right)_2 \partial_x\phi_k &= \left( \chi^{\prime\prime}_\delta(f_\e)(t)\partial_x f_\e - \displaystyle{\sum_{k=0}^n}\left( \chi^\prime_\delta(f_\e(t)) \mid \phi_k \right)_2 \partial_x\phi_k\right)\\[5pt]
&\qquad + \displaystyle{\sum_{k=0}^n}  \left( (\chi_\delta^\prime(f_\e)(t)-\chi_\delta^\prime(f_\e^n)(t) )\mid \phi_k \right)_2 \partial_x\phi_k.
\end{split}
\end{align}
The composition $\chi^\prime_\delta(f_\e(t))$ belongs to $ H^1(0,L)$ and possesses a Fourier expansion with
\[ \displaystyle{\sum_{k=0}^n}\left( \chi_\delta^\prime(f_\e^n(t)) \mid \phi_k \right)_2 \phi_k \longrightarrow  \chi_\delta^\prime(f_\e(t))\qquad\mbox{in}\qquad H^1(0,L).\]
As a consequence, the first term of the right-hand side of \eqref{first} converges to zero in $L_2(\Omega_T)$.
Concerning the convergence of the second term in \eqref{first}, note that the sum is the truncation function of the Fourier expansion of $\chi^{\prime\prime}_\delta(f_\e)\partial_x f_\e-\chi^{\prime\prime}_\delta(f_\e^n)\partial_x f_\e^n$ and may be estimated as follows
\begin{align*}
\left\|\displaystyle{\sum_{k=0}^n}  \left( (\chi_\delta^\prime(f_\e)-\chi_\delta^\prime(f_\e^n) )\mid \phi_k \right)_2 \partial_x\phi_k \right\|_2^2 &\leq \|\chi^{\prime\prime}_\delta(f_\e)\partial_x f_\e-\chi^{\prime\prime}_\delta(f_\e^n)\partial_x f_\e^n\|_2^2 \\[5pt]
&=\|\chi^{\prime\prime}_\delta(f_\e)\partial_x f_\e-\chi^{\prime\prime}_\delta(f_\e^n)\partial_x f_\e+\chi^{\prime\prime}_\delta(f_\e^n)\partial_x f_\e-\chi^{\prime\prime}_\delta(f_\e^n)\partial_x f_\e^n\|_2^2 \\[5pt]
&\leq 2\left( \|\chi^{\prime\prime}_\delta(f_\e)-\chi^{\prime\prime}_\delta(f_\e^n)\|_\infty^2\|\partial_x f_\e\|_2^2+  \|\chi^{\prime\prime}_\delta(f_\e^n)\|_\infty^2\|\partial_x f_\e-\partial_x f_\e^n\|_2^2 \right).
\end{align*}
Since $\chi_\delta^{\prime\prime}=\delta^{-1}\psi(\frac{\cdot}{\delta})$ and $\psi$ is globally Lipschitz continuous, we deduce that 
\[\|\chi^{\prime\prime}_\delta(f_\e)-\chi^{\prime\prime}_\delta(f_\e^n)\|_\infty^2\leq c_1(\delta)  \|f_\e-f_\e^n\|_\infty^2\]
and, in virtue of Lemma \ref{chi} ii),
 \[\|\chi^{\prime\prime}_\delta(f_\e^n)\|_\infty^2\leq c_2(\delta),\]
 for some constants $c_1(\delta), c_2(\delta)>0$, depending on $\delta>0$. Eventually, the estimate reads
\begin{align*}
\left\|\displaystyle{\sum_{k=0}^n}  \left( (\chi_\delta^\prime(f_\e)-\chi_\delta^\prime(f_\e^n) )\mid \phi_k \right)_2 \partial_x\phi_k \right\|_2^2
\leq 2\left(c_1(\delta) \|f_\e-f_\e^n\|_\infty^2\|\partial_x f_\e\|_2^2 +c_2(\delta) \|\partial_x f_\e-\partial_x f_\e^n\|_2^2\right),
\end{align*}
which tends to zero if $n\longrightarrow \infty$, by \eqref{csf} and Lemma \ref{chi} ii). Hence,
\begin{align*}
\displaystyle{\sum_{k=0}^n}\left( \chi_\delta^\prime(f_\e^n)\mid \phi_k \right)_2 \partial_x\phi_k \longrightarrow \chi^{\prime\prime}_\delta(f_\e)\partial_x f_\e \qquad\mbox{in}\qquad L_2(\Omega_T).
\end{align*}
Since $(H_f^{\e,n})_{n\in \N}$ converges weakly to $H_f^\e$ in $L_2(\Omega_T)$, Lemma \ref{App} implies
that we can pass to the limit in the second term of \eqref{42} as well, which yields the first statement in \eqref{39}.
The assertion for $g_\e$ in \eqref{39} works similarly, so that the proof is complete.
\end{proof}

The following corollary shows that an accumulation point $(f, g)$ of the sequence  $(f_{\e_k},g_{\e_k})_{\e_k}$ as in \eqref{ap} is non-negative. 
\begin{kor}\label{nonn} Assume that $f^0,g^0\geq0$. Then, an accumulation point  $(f, g) \in (C(\overline{\Omega}_T))^2$ of the sequence of solutions to the regularized systems $(f_{\e_k},g_{\e_k})_{\e_k}$ as in \eqref{ap} is non-negative.
\end{kor}

\begin{proof} Let $(\e_k)_{k\in \N}\in (0,1]$ be such that $\e_k\searrow 0$ for $k\longrightarrow \infty$. Then,
\begin{align*} \|\chi_{\sqrt{\e_k}}(f_{\e_k})-\mbox{max}\{-f,0\}\|_\infty  &\leq \|\chi_{\sqrt{\e_k}}(f_{\e_k})-\chi_{\sqrt{\e_k}}(f)\|_\infty +\|\chi_{\sqrt{\e_k}}(f)-\mbox{max}\{-f,0\}\|_\infty\\[10pt]
&\leq  \|f_{\e_k}-f\|_\infty +\sqrt{\e_k},
\end{align*}
by Lemma \ref{chi} i) and ii). 
Recall that in the previous lemma we have shown that
\[\left|\displaystyle{\int_0^L} \chi_{\sqrt{\e_k}}(f_{\e_k}(t))\, dx \right|\leq c \sqrt{t}\e_k,\qquad \mbox{for all} \quad t\in [0,T],\]
where $c>0$ is a constant independent of $\e\in(0,1]$ and $t\in[0,T]$.
Hence, letting $k$ tend to infinity, implies that \[\displaystyle{\int_0^L}\mbox{max}\{-f(t),0\}\, dx=0\] for all $t\in[0,T]$, which proves the statement for $f$. The non-negativity of $g$ follows by the same argumentation
 \footnote{The proof of Corollary \ref{nonn} is essentially due to Lemma \ref{chilemma}, which provides an estimate depending on $\e$ of the negative part of a function. Remark that we did not claim the non-negativity of $(f_\e,g_\e)_{\e \in (0,1]}$ itself , but only for an accumulation point of this family when $\e \searrow 0$.}.
\end{proof}
Following the idea in \cite{EW3}, we prove in the next theorem that the sequence $(\Gamma_\e)_{\e\in (0,1]}$ already admits the property to be non-negative almost everywhere. 

\begin{thm}\label{nonnG} Assume that $\Gamma^0\geq0$. Then $\Gamma_\e$, $\e \in(0,1]$, is non-negative almost everywhere in $\Omega_T$.
\end{thm}  

\begin{proof} Let $\delta >0$ and $\chi_\delta$ the function defined in \eqref{mollifier}. Then, $\chi_\delta(\Gamma_\e^n(t))\in H^1(0,L)$ for all $t> 0$ and 
\begin{align*}
\frac{d}{dt}\displaystyle{\int_0^L} \chi_\delta(\Gamma_\e^n(t))\, dx = \displaystyle{\int_0^L} \chi^\prime_\delta(\Gamma_\e^n(t))\partial_t \Gamma_\e^n(t)\, dx = \displaystyle{\int_0^L} \partial_t \Gamma_\e^n(t)\displaystyle{\sum_{k=0}^n}\left( \chi^\prime_\delta(\Gamma_\e^n(t))\mid\phi_k\right)_2 \phi_k\, dx,
\end{align*}
which yields after integration with respect to time
\begin{equation}\label{42G}
\displaystyle{\int_0^L} \chi_\delta(\Gamma_\e^n(T))\, dx =  \displaystyle{\int_{\Omega_T}} H_\Gamma^{\e,n} \displaystyle{\sum_{k=0}^n}\left( \chi^\prime_\delta(\Gamma_\e^n(t))\mid \phi_k \right)_2 \partial_x \phi_k \, d(x,t)
\end{equation}
for each $T>0$. We can pass to the limit in \eqref{42G}, by the same argument as in the proof of Lemma \ref{chilemma}, and obtain\footnote{Recall that introducing the truncation function $\tau_\e$ in \eqref{rsystem}, provides that $H_\Gamma^\e$ belongs to $L_2(\Omega_T)$ (instead of $L_{\frac{3}{2}}(\Omega_T)$ cf. \eqref{regGamma}). This improved regularity allows to pass to the limit in \eqref{42G}.}
\begin{equation}\label{61}
\displaystyle{\int_0^L}\chi_\delta(\Gamma_\e(T))\,dx  = \displaystyle{\int_{\Omega_T}} H_\Gamma^{\e} \partial_x \chi^\prime_\delta(\Gamma_\e)\, d(x,t),
\end{equation}
where 
$H_\Gamma^\e$ represents the limit of a weakly convergent subsequence of $H_\Gamma^{\e,n}$ in $L_2(\Omega_T)$.
By construction it is $\chi^{\prime\prime}_\delta = 0$ on $\R\setminus [-\delta,0]$, so that \eqref{61} yields
\begin{align*}
\displaystyle{\int_0^L}\chi_\delta(\Gamma_\e(T))\,dx  &= \displaystyle{\int_{[-\delta \leq \Gamma_\e \leq 0]}} \left\{ \displaystyle{\frac{\sqrt{3}}{2}}\tau_\e(\Gamma_\e) \sqrt{a_\e(f_\e)}J_{f,g}^\e+\displaystyle{\frac{\sqrt{3}}{2}}\tau_\e(\Gamma_\e) \sqrt{a_\e(g_\e)}J_g^\e\right.\\[5pt]
&\left.\qquad\qquad+\frac{1}{4}\tau_\e(\Gamma_\e) a_\e(g_\e)\partial_x \sigma_\e(\Gamma_\e)-D\partial_x\Gamma_\e \right\} \chi^{\prime\prime}_\delta(\Gamma_\e)\partial_x  \Gamma_\e\, d(x,t)\\[5pt]
  &\leq  \displaystyle{\int_{[-\delta \leq \Gamma_\e \leq 0]}} \left\{ \displaystyle{\frac{\sqrt{3}}{2}}\tau_\e(\Gamma_\e) \sqrt{a_\e(f_\e)}|J_{f,g}^\e|+\displaystyle{\frac{\sqrt{3}}{2}}\tau_\e(\Gamma_\e) \sqrt{a_\e(g_\e)}|J_g^\e|\right.\\[5pt]
&\left.\qquad\qquad+\frac{1}{4}\tau_\e(\Gamma_\e) a_\e(g_\e)|\partial_x \sigma_\e(\Gamma_\e)| \right\} \chi^{\prime\prime}_\delta(\Gamma_\e)\partial_x  \Gamma_\e\, d(x,t),
\end{align*}
where we used the fact that $\chi^{\prime\prime}_\delta = \delta^{-1}\psi(\frac{\cdot}{\delta})\geq 0$, which implies $-D\chi^{\prime\prime}_\delta(\Gamma_\e)|\partial_x \Gamma_\e|^2\leq 0$. By means of $|\tau_\e(s)\chi_\delta^{\prime\prime}(s)|\leq |s\chi_\delta^{\prime\prime}(s)|\leq K$ if $|s|\leq \delta$ (cf. Lemma \ref{chi} iii)), we find that
\begin{align*}
\displaystyle{\int_0^L}\chi_\delta(\Gamma_\e(T))\,dx & \leq \displaystyle{\frac{\sqrt{3}}{2}}\| \sqrt{a_\e(f_\e)} \|_\infty  K\displaystyle{\int_{[-\delta \leq \Gamma_\e \leq 0] }}  \left|J_{f,g}^\e\partial_x  \Gamma_\e \right|  \,d(x,t) \\[5pt]
&\quad+\displaystyle{\frac{\sqrt{3}}{2}}\left\|\sqrt{a_\e(g_\e)}\right\|_\infty  K\displaystyle{\int_{[-\delta \leq \Gamma_\e \leq 0]}}\left|J_g^\e\partial_x  \Gamma_\e^2\right| \,d(x,t) \\[5pt]
  &\quad+\frac{1}{4}\left\|\sqrt{a_\e(g_\e)}\right\|_\infty K\displaystyle{\int_{[-\delta \leq \Gamma_\e \leq 0]}}\left|  \sqrt{a_\e(g_\e)}\sigma_\e^\prime(\Gamma_\e)     \partial_x\Gamma_\e \right|  \,d(x,t).
\end{align*}
By Hölder's inequality, the estimate implied by the energy equality \eqref{energy}, the bound of $(\partial_x \Gamma_\e)_{\e \in (0,1]}$ in $L_2(\Omega_T)$ and the definition of $a_\e$ together with $(f_\e)_{\e\in (0,1]}, (g_\e)_{\e\in (0,1]}$ being bounded in $L_\infty(\Omega_T)$ (cf. Remark \ref{fepsu}), the above inequality implies that
\begin{align*}
\displaystyle{\int_0^L}\chi_\delta(\Gamma_\e(t))\,dx  \leq c \displaystyle{\int_{[-\delta \leq \Gamma_\e \leq 0]}}\left| \partial_x\Gamma_\e \right|^2  \,d(x,t)
\end{align*}
for some constant $c>0$. It follows from \cite[Lemma A.4]{KS} that for almost all $t\geq 0$
\begin{equation*}
\displaystyle{\int_0^L} \max\, \{-\Gamma_\e(t),0\}\, dx = \lim_{\delta\rightarrow 0}\displaystyle{\int_0^L}\chi_\delta(\Gamma_\e(t))\,dx  \leq 0,
\end{equation*}
which completes the proof.
\end{proof}

\subsection{Existence of Weak Solutions to the Original Problem.}

Now, we prove that there exists indeed an accumulation point of the family of solutions to the regularized system $(f_\e,g_\e,\Gamma_\e)_{\e\in (0,1]}$ being a global weak solution to the original problem \eqref{system2}.
To start with, recall that Lemma \ref{lffg}, Remark \ref{fepsu}, Lemma \ref{lfG}, \eqref{ee5} and \eqref{ee2}--\eqref{eE5} provide the following bounds\footnote{Keep in mind that the bounds $\partial_x^3 f_\e^n$ and $\partial_x^3 g_\e^n$ in $L_2(\Omega_T)$ are not uniform in $\e\in (0,1]$ and we loose these regularities, when passing to the limit $\e\searrow 0$.} :
\allowdisplaybreaks 
\begin{alignat}{2}
\label{37}&\{f_\e, g_\e\mid {\e \in (0,1]}\} &&\mbox{in} \; L_\infty(0,T;H^1(0,L)), \\[5pt]
\label{37a}&\{\partial_t f_\e, \partial_t g_\e\mid {\e \in (0,1]}\} &&\mbox{in} \; L_2(0,T,(H^1(0,L))^\prime), \\[5pt]
\label{JF} & \{J_f^\e\mid {\e \in (0,1]} \} \quad&& \mbox{in} \; L_2(\Omega_T),\\[5pt]
\label{JFG} & \{J_{f,g}^\e\mid {\e \in (0,1]} \} \quad&& \mbox{in} \; L_2(\Omega_T),\\[5pt]
\label{JG} & \{J_g^\e\mid {\e \in (0,1]} \} \quad&& \mbox{in} \; L_2(\Omega_T),\\[5pt]
\label{xx} &\{ \sqrt{a_\e(g_\e)}\partial_x \sigma_\e(\Gamma_\e)\mid \e \in (0,1]\} \qquad&& \mbox{in} \; L_2(\Omega_T).
\end{alignat}
By the same arguments used before, we find a sequence $(\e_k)_{k\in \N} \in (0,1]$ with $\e_k \searrow 0$, such that
\begin{alignat}{2}
\label{34}
&f_{\e_k} \longrightarrow f \qquad \mbox{and}\qquad g_{\e_k} \longrightarrow g \qquad &&\mbox{in}\quad C([0,T],C^\alpha([0,L])),\\[5pt]
\label{35}
&f_{\e_k} \warrow f\qquad \mbox{and}\qquad g_{\e_k} \warrow g \qquad &&\mbox{in}\quad L_2(0,T;H^1([0,L])),\\[5pt]
\label{35a}
&\partial_t f_{\e_k}\warrow \partial_t f\qquad \mbox{and}\qquad \partial_t g_{\e_k}\warrow \partial_t g \qquad &&\mbox{in}\quad L_2(0,T;(H^1([0,L]))^\prime),
\end{alignat}
for $\alpha\in [0,\frac{1}{2})$. 
In particular, after possibly extracting a further subsequence, we obtain that 
\begin{equation*}
\partial_x f_{\e_k}(t) \warrow \partial_ x f(t)\qquad \mbox{and}\qquad \partial_x g_{\e_k}(t) \warrow \partial_x g(t) \quad \mbox{in}\quad L_2(0,L)
\end{equation*}
for almost all $t\in [0,T]$ and 
\begin{equation}\label{51}
f,g \in L_\infty (0,T,H^1(0,L)) \cap C([0,T],C^\alpha([0,L])).
\end{equation}
Recall that (cf. Lemma \ref{lfG})
\begin{align*}
&\{\Gamma_{\e_k}\mid {\e_k} \in (0,1]\}\quad \mbox{is bounded in} \quad L_\infty(0,T;L_2(0,L))\cap L_{2}(0,T;H^1(0,L)),\\
&\{\partial_t \Gamma_{\e_k}\mid {\e_k}\in(0,1]\}\quad \mbox{is bounded in} \quad L_{\frac{3}{2}}(0,T;(W^1_3(0,L))^\prime),
\end{align*}
which implies in view of \cite[Corollary 4]{Sim} that 
\[(\Gamma_{\e_k})_{{\e_k}\in (0,1]} \quad \mbox{is relatively compact in} \quad C([0,T]; (W^1_3(0,L))^\prime) \cap L_2(0,T;C^\alpha([0,L])),\]
for $\alpha\in [0,\frac{1}{2})$. Hence, there exists a subsequence (not relabeled) such that
\begin{equation}\label{36a}\Gamma_{\e_k} \longrightarrow \Gamma \qquad \mbox{in} \quad C([0,T]; (W^1_3(0,L))^\prime) \cap L_2(0,T;C^\alpha([0,L])).
\end{equation}
Similar as in the previous section, we deduce that the limit function $\Gamma$ satisfies
\begin{equation}\label{limit_Gamma}
\Gamma \in L_\infty(0,T;L_2(0,L))\cap L_{2}(0,T;H^1(0,L)) \cap C([0,T]; (W^1_3(0,L))^\prime).
\end{equation}
Furthermore,
\begin{alignat}{2}
\label{36}
&\Gamma_{\e_k} \warrow \Gamma \qquad  &&\mbox{in}\quad L_2(0,T;H^1(0,L)),\\[5pt]
\label{36b}
&\partial_t \Gamma_{\e_k}\warrow \partial_t \Gamma\qquad  &&\mbox{in}\quad L_\frac{3}{2}(0,T;(W^1_3(0,L))^\prime),\\[5pt]
\label{36d}
&\Phi(\Gamma_{\e_k}) \longrightarrow \Phi(\Gamma) \qquad  &&\mbox{in}\quad L_1(\Omega_T),
\end{alignat}
where the last assertion can be proved analogously to Lemma \ref{lfG} iv).
Thus, by  \eqref{51} and \eqref{limit_Gamma} we have shown the regularity for $f,g$ and $\Gamma$ claimed by Theorem \ref{MT} a).
In virtue of Corollary \ref{nonn}, the functions $f$ and $g$ are non-negative, whereas  $\Gamma\geq 0$ almost everywhere in view of Theorem \ref{nonnG} and \eqref{36a}.
Further, $f(0)=f^0$, $g(0)=g^0$ point-wise and $\Gamma(0)=\Gamma^0$ almost everywhere, by \eqref{inC}, \eqref{34} and \eqref{36a}. Therefore claim b) of Theorem \ref{MT} is satisfied.  Due to \eqref{mass}, \eqref{34} and \eqref{36a}, the conservation of mass property in c) of Theorem \ref{MT} is satisfied. 

Last, we establish the identities in Theorem \ref{MT} d). In order to be able to pass to the limit in \eqref{T1}--\eqref{T3}, we investigate, like in \cite[Proof of Theorem 3]{EW3}, the convergence of the regularized terms $\tau_\e$ and $\sigma_\e$, which occur in $H_f^\e, H_g^\e$ and $H_\Gamma^\e$. Note first that (as in Lemma \ref{lfG}), we find that $(\Gamma_\e)_{\e \in (0,1]}$ is bounded in $L_6(\Omega_T)$ and the convergence $\Gamma_\e \longrightarrow \Gamma$ takes place in $L_p(\Omega_T)$ for $p\in [1,6)$. Moreover, by construction
\begin{align}\label{taus}
\tau_\e(s)=s \qquad \mbox{for}\quad 0\leq s\leq s_\e := \left[ \left(\frac{1}{\e C_\Phi}\right)^{\frac{r}{r+1}}-1 \right]^{\frac{1}{r}},
\end{align}
which is due to Assumption A3).
In particular, we obtain that
\begin{equation}\label{sigmae}
\sigma_\e^\prime(s)= \sigma^\prime(s)\qquad \mbox{for all} \quad s\in [0,s_\e].
\end{equation}

\begin{lem}\label{allc} There exists a subsequence (not relabeled) of $(\Gamma_\e)_{\e\in (0,1]}$ satisfying
\begin{itemize}
\item[i) ] $\tau_\e(\Gamma_\e)\longrightarrow \Gamma$ in $L_q(\Omega_T)$ for $q\in [1,6)$,
\item[ii) ] $\partial_x \sigma_\e(\Gamma_\e)\warrow \partial_x \sigma(\Gamma)$ in $L_s(\Omega_T)$ for $s\in [1,\frac{6}{5})$. 
\end{itemize}
\end{lem}

\begin{proof} Recall that $\Gamma_\e\geq 0$ almost everywhere, due to Theorem \ref{nonnG}.

 i) We show first that 
\begin{align}\label{ctau}
\frac{\tau_\e(\Gamma_\e)}{\Gamma_\e}\longrightarrow 1 \qquad \mbox{in}\quad L_p(\Omega_T) \quad \mbox{for any } \quad p\geq 1.
\end{align} 
Then, the statement follows in virtue of
\begin{align*}
\|\tau_\e(\Gamma_\e)-\Gamma\|_p \leq \|\tau_\e(\Gamma_\e)-\Gamma_\e\|_p+\|\Gamma_\e-\Gamma\|_p \leq \left\|\Gamma_\e\right\|_6\left\|\frac{\tau_\e(\Gamma_\e)}{\Gamma_\e}-1\right\|_{\frac{6p}{6-p}} +\left\|\Gamma_\e-\Gamma\right\|_p,
\end{align*}
$\Gamma_\e \rightarrow \Gamma$ in $L_m(\Omega_T)$ for $m\in [1,6)$ and \eqref{ctau}.
In order to prove \eqref{ctau}, recall that $\tau_\e(\Gamma_\e)= \Gamma_\e$ if $\Gamma_\e \leq s_\e$. Thus, for any $p\geq 1$, there exists a constant $C>0$, such that
\begin{align}\label{tauu}
\begin{split}
\displaystyle{\int_{\Omega_T}}\left|\frac{\tau_\e(\Gamma_\e)}{\Gamma_\e}-1  \right|^p\,d(x,t)  &= \displaystyle{\int_{[\Gamma_\e >s_\e]}}\left|\frac{\tau_\e(\Gamma_\e)}{\Gamma_\e}-1  \right|^p\,d(x,t) \leq \displaystyle{\int_{[\Gamma_\e >s_\e]}}2^p\left( \left|\frac{\tau_\e(\Gamma_\e)}{\Gamma_\e}\right|^p-1 \right) \,d(x,t) \\[5pt]
& \leq 2^{p+1}\displaystyle{\int_{[\Gamma_\e >s_\e]}} 1 \,d(x,t) \leq 2^{p+1}\displaystyle{\int_{[\Gamma_\e >s_\e]}} \frac{\Gamma_\e^{6}}{s_\e^{6}} \,d(x,t)\leq \frac{C}{s_\e^{6}}, 
\end{split}
\end{align}
since $|\tau_\e(s)|\leq |s|$ and $(\Gamma_\e)_{\e \in (0,1]}$ being uniformly bounded in $L_6(\Omega_T)$. Letting $\e$ tend to zero, \eqref{tauu} implies the assertion in view of $s_\e\longrightarrow \infty$ if $\e \searrow 0$.

ii) Given $p\in [1,\frac{6}{r+1})$, $R\geq 1$ and $\e\in (0,1]$, such that $1\leq R\leq s_\e$, we have that
\begin{align}\label{sig}
\begin{split}
\displaystyle{\int_{\Omega_T}}|\sigma_\e^\prime(\Gamma_\e)-\sigma^\prime(\Gamma)|^p\,d(x,t) =&\displaystyle{\int_{[\max\{\Gamma_\e,\Gamma\} \leq R]}}|\sigma_\e^\prime(\Gamma_\e)-\sigma^\prime(\Gamma)|^p\,d(x,t)\\[5pt]
&+\displaystyle{\int_{[\Gamma_\e >R]\cup [\Gamma>R]}}|\sigma_\e^\prime(\Gamma_\e)-\sigma^\prime(\Gamma)|^p\,d(x,t).
\end{split}
 \end{align}
Estimating the integrals on the right-hand side of \eqref{sig} separately, noticing that $\sigma_\e^\prime = \sigma^\prime$ everywhere in $[\Gamma_\e \leq R]$ (cf. \eqref{sigmae}) and since $\sigma^\prime \in C^1(\R)$, the Mean Value Theorem implies that the first integral reduces to
\begin{align}\label{sig1}
\begin{split}
\displaystyle{\int_{[\max\{\Gamma_\e,\Gamma\} \leq R]}}|\sigma_\e^\prime(\Gamma_\e)-\sigma^\prime(\Gamma)|^p\,d(x,t) &= \displaystyle{\int_{[\max\{\Gamma_\e,\Gamma\} \leq R]}}|\sigma^\prime(\Gamma_\e)-\sigma^\prime(\Gamma)|^p\,d(x,t) \\[5pt]
& \leq \|\sigma^{\prime\prime}\|_{L_{\infty}(0,R)} \displaystyle{\int_{[\max\{\Gamma_\e,\Gamma\} \leq R]}}|\Gamma_\e-\Gamma|^p\,d(x,t),
\end{split}
\end{align}
which tends to zero if $\e\searrow 0$ for any $p\in [1,6)$. 
The second integral yields in virtue of $|\sigma_\e^\prime|\leq |\sigma^\prime|$ and Assumption A3)
\begin{align}\label{sig3}
\begin{split}
\displaystyle{\int_{[\Gamma_\e >R]\cup [\Gamma>R]}}|\sigma_\e^\prime(\Gamma_\e)-\sigma^\prime(\Gamma)|^p\,d(x,t) &\leq \displaystyle{\int_{[\Gamma_\e >R]\cup [\Gamma>R]}}2^p \left(|\sigma^\prime(\Gamma_\e)|^p+|\sigma^\prime(\Gamma)|^p\right)\,d(x,t)\\[5pt]
& \leq 2^{p}C_{\Phi}\displaystyle{\int_{[\Gamma_\e >R]\cup [\Gamma>R]}} |\Gamma_\e(\Gamma_\e^r+1)|^p+|\Gamma(\Gamma^r+1)|^p\,d(x,t) \\[5pt]
&\leq 2^{p+1}C_{\Phi}\displaystyle{\int_{[\Gamma_\e >R]\cup [\Gamma>R]}} 2\max\{\Gamma_\e,\Gamma\}^{p(r+1)}\,d(x,t) \\[5pt]
&= \frac{2^{p+3}C_{\Phi}}{R^{6-p(r+1)}}\displaystyle{\int_{[\Gamma_\e >R]\cup [\Gamma>R]}}  \max\{\Gamma_\e,\Gamma\}^{p(r+1)} R^{6-p(r+1)}\,d(x,t) \\[5pt]
&\leq \frac{2^{p+2}C_{\Phi}}{R^{6-p(r+1)}}\displaystyle{\int_{[\Gamma_\e >R]\cup [\Gamma>R]}} \Gamma_\e^6+\Gamma^6\,d(x,t).
\end{split}
\end{align}
Now, we may let first $\e \searrow 0$ and then $R\rightarrow \infty$ in \eqref{sig3}. Gathering \eqref{sig}--\eqref{sig3}, we have shown that
\begin{equation}\label{Sig}
\sigma_\e^\prime(\Gamma_\e)\longrightarrow \sigma^\prime(\Gamma)\qquad \mbox{in}\quad L_2(\Omega_T).
\end{equation}
Recalling that $(\partial_x \sigma_\e(\Gamma_\e))_{\e\in (0,1]} $ is bounded in $L_s(\Omega_T)$ for $s\in [1,\frac{6}{5})$ and  $(\partial_x \Gamma_\e)_{\e\in(0,1]}$ being bounded in $L_2(\Omega_T)$, the statement follows then by Lemma \ref{App},  \eqref{36} and \eqref{Sig}.
\end{proof}
Let $(\e_k)_{k\in\N}\subset (0,1]$ be a sequence tending to zero, when $k\longrightarrow \infty.$ Moreover,
let $\xi\in L_2(0,T;H^1(0,L))$ be given and  $(f_{\e_k}, g_{\e_k},\Gamma_{\e_k})_{k\in\N}$ be the family of solutions to the regularized system, which admits a subsequence converging towards $(f,g,\Gamma)$. 
Studying the convergence of
\begin{align}\label{T1MM}
\displaystyle{\int_0^T}  \langle \partial_t f_{\e_k}(t), \xi(t) \rangle_{H^1} \,dt= \displaystyle{\int_{\Omega_T}}  H_f^{\e_k}\partial_x \xi \,d(x,t),
\end{align}
we observe first, that $(H_f^{{\e_k}})_{k\in\N}$ is bounded in $L_2(\Omega_T)$ in view of \eqref{34} and the energy inequality \eqref{energy} being satisfied for weak solutions of the regularized system. Thus, by Eberlein--Smulyan's theorem, there exists a weakly convergent subsequence (not relabeled), with
\begin{equation}
H_f^{\e_k}\warrow H_f\qquad \mbox{in}\quad L_2(\Omega_T)
\end{equation}
and the right-hand side of \eqref{T1MM} converges to the desired equation in \eqref{MT11}. The convergence of the left-hand side of \eqref{T1MM} is due to \eqref{35a}. Similarly one proves \eqref{MT22} and \eqref{MT3}. 
We show, that the function $H_f$ can be identified with $\frac{f^{\frac{3}{2}}}{\sqrt{3}}J_f$ on the set $\mathcal{P}_f \cap \mathcal{P}_g$ as claimed in Theorem \ref{MT} d).
Define the sets 
\[\mathcal{P}_f^m:=\left\{(t,x)\in\Omega_T:\;f(t,x)>\frac{1}{m}\right\}\qquad \mbox{and}\qquad \mathcal{P}_g^m:=\left\{(t,x)\in\Omega_T:\; g(t,x)>\frac{1}{m}\right\}.\]
Then, $\mathcal{P}_f\cap \mathcal{P}_g= \bigcap_{m=1}^\infty \left(\mathcal{P}_f^m\cap \mathcal{P}_g^m \right)$.
We deduce form the continuous convergences in \eqref{34}, \eqref{35} that for every $m\in \N$ there exists $k_0\in \N$, such that 
\[f_{\e_k}>\frac{1}{2m},\quad g_{\e_k}>\frac{1}{2m}\qquad \mbox{for all} \quad k\geq k_0\quad \mbox{and}\quad (t,x)\in \mathcal{P}_f^m\cap \mathcal{P}_g^m.\]
In view of \eqref{37}, \eqref{JF}, \eqref{JG} and $R>S$, the sequences $(\partial_x^3 f_{\e_k})_{k\geq k_0}$ and $(\partial_x^3 g_{\e_k})_{k\geq k_0}$ are bounded in $L_2(\mathcal{P}_f^m\cap \mathcal{P}_g^m)$ and
\begin{equation}\label{3bound}
\partial_x^3 f_{\e_k}\warrow \partial_x^3 f,\qquad \partial_x^3 g_{\e_k}\warrow \partial_x^3 g\qquad\mbox{in}\quad L_2(\mathcal{P}_f^m\cap \mathcal{P}_g^m)
\end{equation}
for all $m\geq 1$. Concluding, thanks to Lemma \ref{allc}, \eqref{JF}, \eqref{34}  and \eqref{3bound}, there exists a subsequence (not relabeled), such that
\[H_f=\frac{f^{\frac{3}{2}}}{\sqrt{3}}J_f\qquad \mbox{on} \quad \mathcal{P}_f \cap \mathcal{P}_g.\]
Analogously, one proves the corresponding identities appearing in Theorem \ref{MT} d).
Eventually, similar as before, we pass to the limit in the energy inequality \eqref{energy} and obtain claim Theorem \ref{MT} e). 

\section{Appendix}

We state a lemma, which ensures the identification of weak limits of product sequences with the product of the corresponding limits of its factors.
\begin{lem}\label{App}
Let $\Omega\subset \R^n$ and $1< p,q,r< \infty$ with $p,q$ being a dual pair. If  $(f_n)_{n\in \N}$, $(g_n)_{n\in \N}$  are sequences satisfying
\begin{itemize}
\item[i) ]  $(f_ng_n)_{n\in \N}$ being bounded in $L_r(\Omega)$,
\item[ii) ] $f_n \rightarrow f \in L_p(\Omega)$ and $g_n \warrow g \in L_q(\Omega)$,
\end{itemize}
then there exists a weakly convergent subsequence (not relabeled) with
\[f_ng_n \warrow fg \qquad \mbox{in}\quad L_r(\Omega).\]
\end{lem}
\begin{proof}
Since $L_r(\Omega)$ is a reflexive Banach space,  Eberlein--Smulyan's theorem implies that there exists a weakly convergent subsequence (not relabeled) such that
\begin{equation}\label{wcon}
f_ng_n \warrow v \qquad \mbox{in}\quad L_r(\Omega),
\end{equation}
 where $v\in L_r(\Omega)$ is the limit function, which we show to coincide with $fg$.
For all $\xi\in C_c^\infty(\Omega)$ we have
\begin{align}\label{identify}
\int_\Omega (fg-v)\xi \,dx = \int_\Omega (f-f_n)g\xi \,dx +\int_\Omega (g-g_n)f_n\xi\,dx +\int_\Omega (g_nf_n-v)\xi\,dx.
\end{align}
The last integral in \eqref{identify} converges to zero in view of \eqref{wcon}. Recalling that the strong convergence of $(f_n)_{n\in \N}$ yields in particular, that $(f_n)_{n\in \N}$ is bounded in $ L_p(\Omega)$, the convergence to zero of the two remaining integrals on the right hand side of \eqref{identify} is a consequence of ii).
\end{proof}

\section*{Acknowledgments}
I am grateful to Joachim Escher and Christoph Walker for proposing this topic of research and for various helpful discussions. This work was supported  by the Deutsche Forschungsgemeinschaft (DFG) (Graduiertenkolleg GRK 1463 Analysis, Geometry and Stringtheory).

\end{document}